\documentclass{ip-journal} 
\usepackage{graphicx}
\usepackage{amssymb,amsmath}
\usepackage{amsthm}
\usepackage{color,graphicx}
\usepackage{hyperref}
\usepackage{color}
\usepackage{mathabx}
\usepackage{subfigure}

\usepackage{subfigure} 

\usepackage{verbatim} 

\usepackage{relsize} 

\newcommand{\be}{\begin{equation}}

\newcommand{\ee}{\end{equation}}

\newcommand{\cn}{{\rm \,cn}}
\newcommand{\sn}{{\rm \,sn}}
\newcommand{\dn}{{\rm \,dn}}
\newcommand{\sech}{{\rm \,sech}}

\newcommand{\Ker}{{\rm \,Ker}}
\newcommand{\K}{{\rm \,K}}
\newcommand{\E}{{\rm \,E}}

\numberwithin{equation}{section}
\numberwithin{figure}{section}
\renewcommand{\theequation}{\thesection.\arabic{equation}}
\newtheorem{theorem}{Theorem}[section]

\newtheorem{remark}[theorem]{Remark}
\newtheorem{lemma}[theorem]{Lemma}

\newtheorem{definition}[theorem]{Definition}

\begin{document}
\vglue-1cm \hskip1cm
\title[Multiple periodic waves of the Schrödinger system with cubic nonlinearity]{Spectral Stability of multiple periodic waves for the Schrödinger system with cubic nonlinearity}

\begin{center}



\maketitle

%
%

{\bf Fábio Natali}

{Departamento de Matem\'atica - Universidade Estadual de Maring\'a\\
	Avenida Colombo, 5790, CEP 87020-900, Maring\'a, PR, Brazil.}\\
{ fmanatali@uem.br}

{\bf Gabriel E. Bittencourt Moraes}

{Departamento de Matem\'atica - Universidade Estadual de Maring\'a\\
	Avenida Colombo, 5790, CEP 87020-900, Maring\'a, PR, Brazil.}\\
{ pg54546@uem.br}

\vspace{3mm}

\end{center}

\begin{abstract}
Results concerning the existence and spectral stability/instability of multiple periodic standing wave solutions for a cubic nonlinear Schrödinger system will be shown in this manuscript. Our approach considers periodic perturbations that have the same period of the standing wave solution. To obtain the quantity and multiplicity of non-positive eigenvalues for the corresponding linearized operator, we use the comparison theorem and tools of Floquet theory. The results are then obtained by applying the spectral stability theory via Krein signature as established in \cite{KapitulaKevrekidisSandstedeI} and \cite{KapitulaKevrekidisSandstedeII}.
\end{abstract}

\section{Introduction} 

In this paper, we show the spectral stability of periodic standing waves for the cubic nonlinear Schrödinger system (NLS-system)
\begin{equation}\label{NLS-system}
	\left\{ \begin{array}{l}
		i u_t + u_{xx} + \kappa_1 |u|^2 u + \gamma v^2 \overline{u} = 0 \\
		i v_t + v_{xx} + \kappa_2 |v|^2 v + \gamma u^2 \overline{v} = 0.
	\end{array}
\right.
\end{equation}
Here $u,v: \mathbb{R} \times (0,+\infty) \rightarrow \mathbb{C}$ are complex-valued functions that are $L$-periodic in the spatial variable, $\kappa_1$ and $\kappa_2$ are positive constants and $\gamma \geq 0$. 

The NLS equation ($v=0$ in $(\ref{NLS-system})$) appears in many applications in physics and engineering as in nonlinear optics, quantum mechanics and nonlinear waves (see \cite{Boyd} and \cite{Fibich}). The cubic nonlinearity has been used as a toy model in field theory. (see \cite{DashenHasslacherNeveu}). The NLS system also appears in nonlinear optics and Bose-Einstein condensates (see \cite{Agrawal, IedaMiyakawaWadati, KannaSakkaravarthi, WangCui}).

For the NLS equation with cubic nonlinearity and periodic boundary conditions, Angulo in \cite{angulo2007} established the orbital stability of periodic standing waves solutions of dnoidal type with respect to perturbations of period $L$. The existence of cnoidal waves was also obtained in \cite{angulo2007}. However, the author did not obtain spectral or orbital stability results in the energy space $H^1_{per}$ for the cnoidal wave. Recently, Natali \textit{et al.} determined in \cite{NataliMoraesLorenoPastor} the orbital stability of cnoidal waves restricted to the subspace $H_{per,m}^1$ of zero mean periodic functions contained in $H_{per}^1$. Employing the approaches in \cite{grillakis1} and \cite{grillakis2}, along a non-degeneracy condition of a suitable $2\times 2$ Hessian matrix, Gallay and H\u ar\u agu\c s in \cite{GallayHaragusJDDE} showed that the cnoidal waves are orbitally stable with respect to semi-periodic perturbations. This work generalized the previous work \cite{GallayHaragusJDE} where the authors used similar arguments to prove the orbital/spectral stability of small amplitude waves with respect to localized or bounded perturbations. Gustafson, Le Coz and Tsai in \cite{GustafsonLeCozTsai} have established spectral stability results for the cnoidal waves with respect to perturbations with the same period $L$ and orbital stability results in the space constituted by anti-periodic functions with period $L/2$. The spectral stability follows by relating the coercivity of the linearized action with the number of eigenvalues with negative Krein signature of a certain operator $J \mathcal{L}$. 

\indent Pastor in \cite{Pastor} considered a similar NLS-system given by
\begin{equation*}
	\left\{ \begin{array}{l}
		i u_t + u_{xx} - u + \left(\frac{1}{9}|u|^2 +2|v|^2\right) u + \frac{1}{3} \overline{u}^2v = 0 \\
		i \sigma v_t + v_{xx} -\alpha v + \left(9|v|^2 +2|u|^2\right) v+ \frac{1}{9}u^3 = 0,
	\end{array}
	\right.
\end{equation*}
where $\alpha$ and $\sigma$ are real constants. First, he proved orbital stability results for dnoidal standing wave solutions using the approaches in \cite{grillakis1} and \cite{grillakis2} by considering periodic perturbations that have the same period of the standing wave solution. Second, he used the theories in \cite{GallayHaragusJDE} and \cite{GallayHaragusJDDE} to demonstrate spectral stability results of periodic waves with respect to localized or bounded perturbations, that is, when the study of the spectrum of a certain
linearized operator $J\mathcal{L}$ is considered over the space $L^2(\mathbb{R})\times L^2(\mathbb{R})$ or $C_b(\mathbb{R})\times C_b(\mathbb{R})$, respectively. Here, $C_b(\mathbb{R})$ indicates the space of (complex) 
continuous functions defined in the whole real line $\mathbb{R}$ that are bounded.

In the context of NLS-system \eqref{NLS-system}, Kawahara and Ohta in \cite{KawaharaOhta} showed the orbital stability and instability of standing solitary wave solutions for the system \eqref{NLS-system}. In that approach, the authors studied the orbital stability properties of semi-trivial standing wave solutions of the form 
\begin{equation}\label{semitrivial-ko}
	(u(x,t),v(x,t)) = (e^{i\omega t} \varphi_\omega(x),0),
\end{equation} where $\varphi_\omega(x) = \sqrt{2 \omega} \sech(\sqrt{\omega} x)$ is a positive and even solution of the equation 
\begin{equation}\label{ODE1}
	-\varphi_{\omega}'' + \omega \varphi_{\omega} - \varphi_{\omega}^3 = 0.
\end{equation}
They proved that the semi-trivial standing wave solution \eqref{semitrivial-ko} is orbitally stable if $\gamma < \kappa_1$ and orbitally unstable if $\gamma > \kappa_1$. In addition, if $\gamma = \kappa_1$, the authors concluded the orbital stability when $\kappa_2 < \kappa_1$ and orbital instability when $\kappa_2 > \kappa_1$. Unfortunately, in the case where $\gamma = \kappa_1 = \kappa_2$, they did not prove the orbital stability (see \cite[Remark 2]{KawaharaOhta}). 

In the periodic setting, Hakkaev in \cite{Hakkaev} studied  the spectral stability of the NLS-system by considering semi-trivial standing waves as in \eqref{semitrivial-ko} where $\varphi_\omega$ is an $L$-periodic function with dnoidal profile which solves equation \eqref{ODE1} and it is given by
\begin{equation*}
	\varphi_\omega(x) = \frac{2 \sqrt{2} {\rm K}(k)}{L} {\rm dn} \left( \frac{2 {\rm K}(k)}{L} x, k\right),
\end{equation*}
where $k \in (0,1)$ is the modulus of the elliptic function and ${\rm K}(k)$ is the complete elliptic integral of the first kind. More precisely, the author showed that the semi-trivial periodic waves are orbitally stable for $\gamma < \kappa_1$. In addition, he also obtained results of spectral stability and instability for the semi-trivial waves. In fact, for $\kappa_1 < \gamma \leq 3\kappa_1$, he concluded that the semi-trivial periodic waves are spectrally unstable and for $\gamma = \kappa_1$ the semi-trivial periodic wave solutions of \eqref{ODE} are spectrally stable.


Now, we give the main topics of our paper. Motivated by \cite{Hakkaev} and \cite{KawaharaOhta}, we consider \textit{multiple periodic standing wave solutions} of \eqref{NLS-system} given by
\begin{equation}\label{multiple}
(u(x,t),v(x,t)) = (e^{i\omega t} \varphi_\omega(x), e^{i\omega t} \vartheta_\omega(x))
:=(e^{i\omega t} \varphi_\omega(x), e^{i\omega t} B \varphi_\omega(x))
\end{equation}
where $\varphi_\omega: \mathbb{R} \rightarrow \mathbb{R}$ is an $L$-periodic function and $\omega \in \mathbb{R}$ is the frequency wave. In addition, we also assume that $B \in \mathbb{R}$ is a real constant which can be assumed non-negative because of the reflection symmetry $v \mapsto -v$.

In our paper, we consider two kind of waves $\varphi=\varphi_\omega$ in the periodic setting. First, we complete the study realized in \cite{Hakkaev} by considering $\varphi$ with (positive) dnoidal profile. Second, equation \eqref{ODE1} has periodic solutions with cnoidal profile that was not mentioned in \cite{Hakkaev}. The cnoidal solution enjoys the zero-mean property, and additional difficulties to apply the spectral stability theories in \cite{KapitulaKevrekidisSandstedeI} and \cite{KapitulaKevrekidisSandstedeII} can arise.
%

Let's start by constructing our periodic solutions. First, by substituting the form \eqref{multiple} into \eqref{NLS-system} we get
\begin{equation}\label{ODE-system}
	\left\{ \begin{array}{l}
		- \omega \varphi + \varphi'' + (\kappa_1 + \gamma B^2) \varphi^3 = 0 \\
		- \omega \varphi + \varphi'' + (\kappa_2 B^2 + \gamma) \varphi^3 = 0.
		\end{array}
	\right.
\end{equation}

In order to determine the existence of multiple solutions, we need to assume that $\kappa_1 + \gamma B^2 = \kappa_2 B^2 + \gamma$. Thus, for $\kappa_2\neq\gamma$, we see that $B>0$ can be expressed by
\begin{equation}\label{B}
	B = \sqrt{ \frac{\kappa_1 - \gamma}{\kappa_2 - \gamma}}.
\end{equation}
In this case and since $B \in \mathbb{R}$, we first consider the three basic cases: 
\begin{equation*}
\gamma \in (0,\min\{\kappa_1,\kappa_2\}), \; \; \gamma \in (\max\{\kappa_1,\kappa_2\},+\infty) \; \; \text{ and } \; \; \gamma = 0.
\end{equation*}
We can also consider the case $\gamma = \kappa_1$ or $\gamma=\kappa_2$ in \eqref{ODE-system}. For both, we conclude that $\gamma = \kappa_1 = \kappa_2$ and $B$ is a free real parameter that does not depend on $\kappa_1$, $\kappa_2$, $\gamma$ and $\omega$.

In all cases mentioned above, the periodic wave $\varphi$ is a solution of the ODE 
\begin{equation}\label{ODE}
	- \varphi'' + \omega \varphi + (\kappa_1 + \gamma B^2) \varphi^3 = 0.
\end{equation}
In the case of solutions with dnoidal profile, we can determine the explicit solution as 
\begin{equation}\label{dn}
	\varphi(x) = \frac{2\sqrt{2} {\rm K}(k)}{L} \frac{1}{(\kappa_1 + \gamma B^2)^{1/2}} {\rm dn}\left( \frac{2 {\rm K}(k)}{L} x, k\right).
\end{equation}
The frequency of the wave $\omega \in \mathbb{R}$ can be expressed as
\begin{equation}\label{omega-dnoidal}
	\omega = \frac{4 (2-k^2) {\rm K}(k)^2}{L^2}.
\end{equation}
By \eqref{omega-dnoidal}, we can see from the dependence of $\omega$ in terms of the parameter $k \in (0,1)$, that $\omega \in \left(\frac{2 \pi^2}{L^2}, +\infty\right)$.\\
\indent  On the other hand, to obtain solutions with cnoidal type, we can proceed similarly as in \cite{angulo2007} to obtain
\begin{equation}\label{cn}
	\varphi(x) = \frac{ \sqrt{2 \omega} k}{\sqrt{2k^2 - 1}} \frac{1}{(\kappa_1 + \gamma B^2)^{1/2}} {\rm cn}\left( \frac{ 4 {\rm K}(k)}{L} x, k \right).
\end{equation}
In this case, the modulus $k$ belongs to the interval $\left(\frac{1}{\sqrt{2}},1\right)$ and the frequency wave $\omega > 0$ is expressed by
\begin{equation}\label{omega-cnoidal}
	\omega = \frac{ 16 {\rm K}(k)^2 (2k^2 - 1)}{L^2}.
\end{equation}
\indent Solution $\varphi$ in \eqref{cn} is an even periodic function. In our spectral stability analysis, it is suitable to work within the complex Sobolev product space $H_{per}^1\times H_{per}^1$ constituted by odd periodic functions. To accomplish this, we must shift the solution $\varphi$ defined in $(\ref{cn})$ by $-\frac{L}{4}$, in order to obtain an odd periodic solution that satisfies equation \eqref{ODE}. In fact, by the formula \cite[Formula 122.05]{byrd}, we deduce,
\begin{equation}\label{expcn1}
	\psi(x)=\varphi\left(x-\frac{L}{4}\right)=\frac{ \sqrt{2 \omega} k\sqrt{1-k^2}}{\sqrt{2k^2 - 1}} \frac{1}{(\kappa_1 + \gamma B^2)^{1/2}} \frac{{\rm sn}\left( \frac{ 4 {\rm K}(k)}{L} x, k \right)}{{\rm dn}\left( \frac{ 4 {\rm K}(k)}{L} x, k \right)}.
\end{equation}
In $(\ref{expcn1})$, the notation ${\rm sn}$ indicates the odd Jacobi elliptic function with snoidal profile.
\begin{remark} For all the solutions mentioned above, we can construct, for each case, a smooth curve of $L-$periodic waves $\omega \in \mathcal{I} \longmapsto \varphi\in H^2_{per}$ that solves \eqref{ODE} (see Theorems $\ref{dnoidalcurve}$ and $\ref{cnoidalcurve}$).
\end{remark}


System \eqref{NLS-system} admits the conserved quantity $E$ defined as
\begin{equation}\label{E}
	E(u,v) = \frac{1}{2} \int_{0}^{L} \left( |u_x|^2 + |v_x|^2 - \frac{\kappa_1}{2} |u|^4 - \frac{\kappa_2}{2} |v|^4 \right) dx - \frac{\gamma}{2} {\rm Re} \int_{0}^L u^2 \overline{v}^2 dx. 
\end{equation}
Moreover, \eqref{NLS-system} has another conserved quantity $F$ given by
\begin{equation}\label{F}
	F(u,v) = \frac{1}{2} \int_{0}^L \left( |u|^2 + |v|^2 \right) dx.
\end{equation}
Then, following similar arguments as in \cite{Cazenave} and using standard fixed point arguments, we can conclude by the conservation laws in \eqref{E} and \eqref{F} that the NLS-system \eqref{NLS-system} is globally well-posed in the complex energy space $H^1_{per}\times H_{per}^1$ (see, for instance, \cite{Colin, Hakkaev, Whitham}).

Now, we present how to obtain the spectral stability of multiple periodic waves with respect to perturbation with the same period. In order to improve the comprehension of the readers, we consider the complex evolution $U=(u,v)$ associated with the system $(\ref{NLS-system})$ of the form,
\begin{equation*}
U(x,t) = (u(x,t),v(x,t)) = ({\rm Re}\,u(x,t), {\rm Re}\, v(x,t), {\rm Im}\, u(x,t), {\rm Im}\, v(x,t)).
\end{equation*}
\indent Consequently, we can consider the stationary solution
 $\Phi = (\varphi, B\varphi,0,0)$ and the perturbation
\begin{equation}\label{U-1}
	U(x,t) = e^{i \beta t} (\Phi(x) + W(x,t))
\end{equation} where $W(x,t) =({\rm Re}\,w_1(x,t), {\rm Re}\, w_2(x,t), {\rm Im}\, w_1(x,t), {\rm Im}\, w_2(x,t)) \in \mathbb{R}^4$. Substituting \eqref{U-1} into \eqref{NLS-system} and neglecting all the nonlinear terms, we get the following linearized equation:
\begin{equation}\label{spectral}
	\frac{d}{dt} W(x,t) = J \mathcal{L} W(x,t),
\end{equation}
where
\begin{equation}\label{J}
	J = \left( \begin{array}{cccc}
		0 & 0 & 1 & 0 \\
		0 & 0 & 0 & 1 \\
		-1 & 0 & 0 & 0 \\
		0 & -1 & 0 & 0 
		\end{array}
	\right),
\end{equation} and $\mathcal{L}$ is the operator given by
\begin{equation}\label{L-1}
	\mathcal{L}= (-\partial_x^2 + \omega){\rm Id} - \varphi^2 S,
\end{equation}
where ${\rm Id} \in \mathbb{M}_{4 \times 4}(\mathbb{R})$ and $S$ is given by
\begin{equation*}
	S = \left(
	\begin{array}{cccc}
		(3\kappa_1 + \gamma B^2) & 2\gamma B & 0 & 0 \\
		2\gamma B & (3\kappa_2 B^2 + \gamma) & 0 & 0 \\
		0 & 0 & (\kappa_1 - \gamma B^2) & 2\gamma B \\
		0 & 0 & 2\gamma B & (\kappa_2 B^2 - \gamma)
	\end{array}
\right).
\end{equation*}

To define the concept of spectral stability within our context, we need to consider $W(x,t)=e^{\lambda t}w(x)$ in the linear equation \eqref{spectral} to obtain the following spectral problem
\begin{equation*}
	J \mathcal{L} w = \lambda w.
\end{equation*} 
\indent The definition of spectral stability in our context reads as follows.

\begin{definition}\label{def-spectralstability}
	The stationary wave $\Phi$ is said to spectrally stable by periodic perturbations that have the same period as the standing wave solution if $\sigma(J \mathcal{L}) \subset i \mathbb{R}$. Otherwise, if there exists at least one eigenvalue $\lambda$ associated with the operator $J \mathcal{L}$ that has a positive real part, $\Phi$ is said to be spectrally unstable.
\end{definition}

As far as we know, it is more convenient to work with the operator $\mathcal{L}$ in a diagonal form. To do so, we need to obtain the existence of an  orthogonal matrix $R$  and a matrix $M$ such that
\begin{equation}\label{S}
	S:= R M R^{-1},
\end{equation}
where $R$ is defined as 
\begin{equation}\label{U-Bfixed}
	R = \left( \begin{array}{cccc}
		-\kappa_2 + \gamma & \frac{1}{2\gamma -\kappa_1 - \kappa_2} & 0 & 0 \\
		\sqrt{(\kappa_1-\gamma)(\kappa_2-\gamma)} & - \frac{1}{2 \gamma - \kappa_1 - \kappa_2}  \frac{1}{B} & 0 & 0 \\
		0 & 0 & - \frac{1}{2\gamma - \kappa_1 - \kappa_2} & \kappa_1 - \gamma \\
		0 & 0 & - \frac{1}{2\gamma - \kappa_1 - \kappa_2} \frac{1}{B} & \sqrt{(\kappa_1-\gamma)(\kappa_2-\gamma)}
	\end{array}
		\right).
\end{equation}
Since $B$ is a real number, we see that $M$ is a matrix with real entries and this fact allows us to deduce that the entries of the matrix $R$ in $(\ref{S})$ are also real numbers (by definition, an orthogonal matrix $R$ is composed of real number entries). The matrix $M \in \mathbb{M}_{4 \times 4}(\mathbb{R})$ is then given by 
\begin{equation*}
	M = \left( \begin{array}{cccc}
		\beta_1 & 0 & 0 & 0 \\
		0 & \beta_3 & 0 & 0 \\
		0 & 0 & \beta_2 & 0 \\
		0 & 0 & 0 & \beta_4
		\end{array}
		\right),
\end{equation*}
where constants $\beta_i \in \mathbb{R}$ are expressed in terms of $\gamma$, $\kappa_1$ and $\kappa_2$ as
\begin{equation}\label{betai}
	\begin{array}{ll}
		\beta_1 = \dfrac{3(\gamma^2 - \kappa_1\kappa_2)}{\gamma - \kappa_2}, & \beta_3 = \dfrac{-\gamma^2 + 2\gamma(\kappa_1 + \kappa_2) - 3\kappa_1\kappa_2}{\gamma - \kappa_2}, \\
		\beta_2 = \dfrac{(\gamma^2 - \kappa_1\kappa_2)}{\gamma-\kappa_2}, & \beta_4 = \dfrac{-3\gamma^2 + 2\gamma(\kappa_1 + \kappa_2) - \kappa_1\kappa_2}{\gamma - \kappa_2}.
	\end{array}
\end{equation} 
Substituting \eqref{S} into \eqref{L-1} and since $R$ is an orthogonal matrix with real entries, we have that $\mathcal{L}$ is a diagonalizable operator with
\begin{equation}\label{ULU}
	\mathcal{L} = R \left( \begin{array}{cccc}
		\mathcal{L}_1 & 0 & 0 & 0 \\
		0 & \mathcal{L}_3 & 0 & 0 \\
		0 & 0 & \mathcal{L}_2 & 0 \\
		0 & 0 & 0 & \mathcal{L}_4
	\end{array}
\right) R^{-1}=R\tilde{\mathcal{L}}R^{-1},
\end{equation}
where $\mathcal{L}_i : H^2_{per} \rightarrow L^2_{per}$ are Hill operators given by
\begin{equation}\label{Li}
	\mathcal{L}_i = -\partial_x^2 + \omega - \beta_i \varphi^2,\ \ \ \ \ i=1,2,3,4.
\end{equation}


It is important to mention that the decomposition in \eqref{ULU} is useful to obtain the non-positive spectrum regarding the operator $\mathcal{L}$ by knowing the non-positive spectrum of $\mathcal{L}_i$, $i=1,2,3,4$, in $(\ref{Li})$. Such decomposition is only possible since solutions in $(\ref{multiple})$ are considered multiple of each other. The existence of non-multiple periodic solutions can be obtained for certain specific parameters $\kappa_1$, $\kappa_2$, and $\gamma$ in equation \eqref{NLS-system}. The challenge lies in achieving spectral stability for this type of waves. Indeed, as it is well-known that in this case, the linearized operator $\mathcal{L}$ in equation \eqref{L-1} cannot be diagonalized, and neither can the entries $V_{jl}$ of the matrix $V$ (see \eqref{V} below).

\indent We now describe our results. Let ${\rm n}(\mathcal{A})$ and ${\rm z}(\mathcal{A})$  be the number of negative eigenvalues and the dimension of the kernel of a certain linear operator $\mathcal{A}$. In our paper, a prior understanding of these non-negative numbers is essential for obtaining the spectral stability result. First, we obtain for the case of dnoidal waves that ${\rm n}(\mathcal{L}_1) = 1$ and ${\rm n}(\mathcal{L}_2) = 0$ (see \cite{angulo2007} and \cite{Hakkaev}). In addition, we have that ${\rm Ker}(\mathcal{L}_1) = [\varphi']$ and ${\rm Ker}(\mathcal{L}_2) = [\varphi]$. An application of the well known comparison theorem in the periodic context (see \cite[Theorem 2.2.2]{Eastham}) gives the behaviour of the non-positive spectrum concerning the operators $\mathcal{L}_3$ and $\mathcal{L}_4$ (see details in Section 4). Next, by considering $\varphi \in H^1_{per}$ with cnoidal profile, we have ${\rm n}(\mathcal{L}_1) = 2$ and ${\rm n}(\mathcal{L}_2) = 1$ (see \cite{angulo2007}). Depending on the choice of the parameters $\gamma$, $\kappa_1$, and $\kappa_2$ in equation $(\ref{NLS-system})$, we cannot conclude a suitable spectral stability result as in the case of dnoidal solutions, since we have too many negative eigenvalues for the operator $\mathcal{L}$. The reason for this is that we cannot apply the comparison theorem to determine the behavior of $\mathcal{L}_3$ and $\mathcal{L}_4$ as we did in the case of dnoidal solutions.\\
\indent To partially overcome this difficulty, we can take advantage of the fact that $\psi=\varphi\left(\cdot-\frac{L}{4}\right)$ is an odd function, and the translated potentials $Q_i=-\beta_i\psi^2$ of the operators $\mathcal{L}_i$ in equation $(\ref{Li})$ are even. Both facts give us that $\mathcal{L}_i$ is well defined in the space $L_{per,odd}^2$ constituted by odd periodic functions in $L_{per}^2$ for all $i=1,2,3,4$. With this information in hands, we can calculate the number of non-positive eigenvalues of $\mathcal{L}_3$ and $\mathcal{L}_4$ within the subspace of odd periodic functions $L_{per,odd}^2$ without further problems. In this setting, we prove that ${\rm n}(\mathcal{L}_{1,odd}) = 1$, ${\rm n}(\mathcal{L}_{2,odd}) = 0$, ${\rm Ker}(\mathcal{L}_{1,odd}) = \{0\}$, and ${\rm Ker}(\mathcal{L}_{2,odd}) = [\varphi]$, where $\mathcal{L}_{i,odd}$ is the restriction operator $\mathcal{L}_i$ defined in $L_{per,odd}^2$ with domain $H_{per,odd}^2$, $i=1,2,3,4$. These facts allow us to use the comparison theorem to obtain the exact behaviour of the non-positive spectrum for the operators $\mathcal{L}_{3,odd}$ and $\mathcal{L}_{4,odd}$ restricted to space $L_{per,odd}^2$. A consequence of this fact is that the study of the spectral stability of periodic waves in the space constituted by odd periodic functions is similar as determined for positive (dnoidal) solutions.


We now obtain our results. To do so, we need to use the methods developed by Kapitula, Kevrekidis and Sandstede in \cite{KapitulaKevrekidisSandstedeI} and \cite{KapitulaKevrekidisSandstedeII}. First, we denote by  $\mathbb{L}^2_{per}$ the space 
\begin{equation*}
\mathbb{L}_{per}^2 = L^2_{per} \times L^2_{per} \times L^2_{per} \times L^2_{per}.
\end{equation*}
If ${\rm z}(\mathcal{L})=n$, consider $\{\Theta_l\}_{l=1}^n \subset {\rm Ker}(\mathcal{L})$ a linearly independent set and let $V$ be the $n\times n$ matrix whose entries are given by 
\begin{equation}\label{V}
V_{jl} = ( \mathcal{L}^{-1} J \Theta_j, J \Theta_l )_{\mathbb{L}^2_{per}},
\end{equation}
where $1\leq j,l\leq n$. The formula
\begin{equation}\label{krein}
	k_r + k_c + k_{-} = {\rm n}(\mathcal{L}) - {\rm n}(V),
\end{equation}
is given in \cite{KapitulaKevrekidisSandstedeII} and the left-hand side of $(\ref{krein})$ is exactly the hamiltonian Krein index, an important tool to decide about the spectral stability and instability of waves. Regarding operator $\mathcal{L}$ in $(\ref{L-1})$, let $k_r$ be the number of real-valued and positive eigenvalues (counting multiplicities). The number $k_c$ denotes the number of complex-valued eigenvalues with a positive real part and $k_-$ is the number of pairs of purely imaginary eigenvalues with negative Krein signature of $\mathcal{L}$. Since $k_c$ and $k_-$ are always even numbers, we obtain that if the right-hand side in \eqref{krein} is an odd number, then $k_r \geq 1$ and we have automatically the spectral instability. Moreover, if the difference ${\rm n}(\mathcal{L}) - {\rm n}(V)$ is zero, then $k_c = k_- = k_r = 0$ which implies the spectral stability. 
	

Summarizing the comments above, our main results concerning the spectral stability of multiple periodic waves of the form $(\ref{multiple})$ are then established:

\begin{theorem}[\textit{Spectral stability/instability for the multiple wave solution with dnoidal profile}]\label{teo-1}
	Let $L>0$ be fixed. Consider the periodic wave solution $\varphi \in H^1_{per}$ of \eqref{ODE} with dnoidal profile given by \eqref{dn}. For $B$ given in \eqref{B} and for all $\omega \in (\frac{2\pi^2}{L^2},+\infty)$, the wave $\Phi = (\varphi, B \varphi,0,0)$ is spectrally unstable if $\gamma \in (0,\min\{\kappa_1,\kappa_2\})$  and spectrally stable if $\gamma \in (\max\{\kappa_1,\kappa_2\},+\infty) \cup \{0\}$. In addition, for $\gamma = \kappa_1 = \kappa_2$ with $B$ being a free real parameter, we obtain that $\Phi$ is spectrally stable for all $\omega \in \left(\frac{2\pi^2}{L^2},+\infty\right).$
\end{theorem}

\begin{theorem}[\textit{Spectral instability for the multiple wave solution with cnoidal profile}]\label{teo-2}
	Let $L>0$ be fixed and consider $\omega>0$. Let $\varphi \in H^1_{per}$ be the  periodic solution of \eqref{ODE} with cnoidal profile given by \eqref{cn}. For $\gamma = \kappa_1 = \kappa_2$ with $B$ being a free real parameter, the wave $\Phi = (\varphi, B \varphi,0,0)$ is spectrally unstable.
\end{theorem}
\indent The translation $\psi$ in $(\ref{expcn1})$ of the periodic cnoidal wave $\varphi$ in $(\ref{cn})$ gives us an odd periodic solution for the equation $(\ref{ODE})$. In this case, it is possible to restrict the spectral stability in the (complex) product space $H_{per,odd}^1\times H_{per,odd}^1$ constituted by odd periodic functions $(e,f)\in H_{per}^1\times H_{per}^1$. The reason for that is that the operator $\mathcal{L}$ in $(\ref{L-1})$, when restricted to $L_{per,odd}^2 \times L_{per,odd}^2$, has a small number of negative eigenvalues. Therefore, it is more convenient to determine whether the difference ${\rm n}(\mathcal{L}) - {\rm n}(V)$ can be zero (indicating stability) or an odd number (indicating instability). Only a few spectral stability scenarios can be determined when $\varphi$ represents the cnoidal solution and the operator $\mathcal{L}$ is considered in either the entire space $L_{per}^2\times L_{per}^2$ or the space $L_{per,even}^2\times L_{per,even}^2$ constituted by even periodic functions $(g,h)\in L_{per}^2\times L_{per}^2$ (see Theorem $\ref{teo-2}$).
\begin{theorem}[\textit{Spectral stability/instability for the multiple wave solution with cnoidal profile and restricted to the subspace of odd functions}]\label{teo-odd}
	Let $L>0$ be fixed and consider $\omega>0$. Let $\psi=\varphi(\cdot-L/4) \in H^1_{per,odd}$ be the periodic solution of \eqref{ODE} with snoidal profile given by \eqref{expcn1}. For $B$ given in \eqref{B}, the multiple wave $\Psi = (\psi, B \psi,0,0)$ is spectrally unstable if $\gamma \in (0,\min\{\kappa_1,\kappa_2\})$ and spectrally stable if $\gamma \in (\max\{\kappa_1,\kappa_2\},+\infty) \cup \{0\}$. In addition, for $\gamma = \kappa_1 = \kappa_2$ with $B$ being a free real parameter, we have that $\Psi$ is spectrally stable.
\end{theorem}

To finish, we give an extension of the results obtained in \cite{Hakkaev} by showing the spectral stability concerning the semi-trivial periodic wave solution
	$(u(x,t),v(x,t)) = (e^{i \omega t} \varphi(x),0),
$
where $\varphi \in H^1_{per}$ has a cnoidal profile. Here, we also analyze the spectral properties concerning the operator $\mathcal{L}$ to obtain the spectral stability results in the same setting of parameters determined in \cite{Hakkaev}. As we have already established for the case of multiple solutions $(\ref{multiple})$, some difficulties appear in the spectral analysis of $\mathcal{L}$. To overcome all these difficulties, we study the case where $\psi=\varphi(\cdot-L/4) $ is given by $(\ref{expcn1})$.

\begin{theorem}[Spectral instability for the semi-trivial wave solution with cnoidal profile] Let $L>0$ be fixed and consider $\omega > 0$. The semi-trivial wave solution $\Psi = (\psi,0,0,0)$ is spectrally unstable when $\gamma = \kappa_1$.
\label{teosemi}	\end{theorem}

\begin{theorem}[Spectral stability/instability for the semi-trivial wave solution with cnoidal profile and restricted to the subspace of odd functions] Let $L>0$ be fixed and consider $\omega>0$. The semi-trivial wave solution $\Psi = (\psi,0,0,0)$ is spectrally stable in $\mathbb{L}_{per,odd}^2$ provided that $\gamma \in (0,\kappa_1]$. In addition, if 
$\gamma \in (\kappa_1,3\kappa_1]$,  the solution $\Psi = (\psi,0,0,0)$ is spectrally unstable in $\mathbb{L}_{per,odd}^2$.
\label{teooddsemi}\end{theorem}


Our paper is organized as follows: 
In Section \ref{section3}, we show the existence of a smooth curve of periodic standing wave solutions of \textit{dnoidal} and \textit{cnoidal} type for the equation \eqref{ODE}. The spectral analysis for the operators $\mathcal{L}$ is determined in Section \ref{section4}. The spectral stability/instability for the multiple wave solutions with dnoidal and cnoidal profile are then established in Section \ref{section5}. Finally, in Section \ref{section6}, we prove the spectral stability/instability result concerning the semi-trivial wave with cnoidal profile.\\


\textbf{Notation:} For $s\geq0$ and $L>0$, the Sobolev space
$H^s_{per}:=H^s_{per}([0,L])$
consists of all periodic functions $f$ such that
$$
\|f\|^2_{H^s_{per}}:= L \sum_{k=-\infty}^{\infty}(1+k^2)^s|\hat{f}(k)|^2 <\infty
$$
where $\hat{f}$ is the periodic Fourier transform of $f$. The space $H^s_{per}$ is a  Hilbert space with the inner product denoted by $(\cdot, \cdot)_{H^s}$. When $s=0$, the space $H^s_{per}$ is isometrically isomorphic to the space  $L^2([0,L])$ and will be denoted by $L^2_{per}:=H^0_{per}$ (see, e.g., \cite{Iorio}). The norm and inner product in $L^2_{per}$ will be denoted by $\|\cdot \|_{L^2_{per}}$ and $(\cdot, \cdot)_{L^2_{per}}$. 

For $s\geq0$, we denote
$$
H^s_{per,odd}:=\{ f \in H^s_{per} \; ; \; f \:\; \text{is an odd function}\}.
$$
endowed with the norm and inner product in $H^s_{per}$.\\
\indent In addition, to facilitate the comprehension of the readers, for $s \geq 0$ and $(f,g) \in H^s_{per} \times H^s_{per}$ (complex), we can write $(f,g)=({\rm Re}\,f, {\rm Re}\,g,{\rm Im}\,f, {\rm Im}\,g)$ and 
$$
\mathbb{H}^s_{per}:= H^s_{per} \times H^s_{per} \times H^s_{per} \times H^s_{per} \; \; \text{ and } \; \; \mathbb{H}^s_{per,odd}:=H^s_{per,odd} \times H^s_{per,odd} \times H^s_{per,odd} \times H^s_{per,odd}
$$
equipped with their usual norms and scalar products.


The symbols $\sn(\cdot, k), \dn(\cdot, k)$ and $\cn(\cdot, k)$ represent the Jacobi elliptic functions of \textit{snoidal}, \textit{dnoidal}, and \textit{cnoidal} type, respectively. For $k \in (0, 1)$, ${\rm F}(\phi, k)$ and $\E(\phi, k)$  denote the complete elliptic integrals of the first and second kind, respectively, and  we denote by $\K(k)={\rm F}\left(\frac{\pi}{2},k\right)$ and $\E(k)=\E\left(\frac{\pi}{2},k\right)$ (for additional details, see \cite{byrd}).


\section{Existence of a Smooth Curve of Periodic Waves}\label{section3}
Our purpose in this section is to present the existence of $L$-periodic solutions $\varphi: \mathbb{R} \longrightarrow \mathbb{R}$ for the following ODE  
\begin{equation}\label{ode1}
- \varphi''+\omega\varphi- (\kappa_1 + \gamma B^2)\varphi^{3}=0,
\end{equation}
where $\omega > 0$.

\subsection{$L$-periodic wave solutions with dnoidal profile}

Consider the ODE
$$
	-\phi'' + \omega \phi - \phi^3 = 0.
$$
By \cite{angulo2007}, we obtain periodic solutions with dnoidal profile as
\begin{equation}\label{psi}
\phi(x)=\frac{2 \sqrt{2} {\rm K}(k)}{L} {\rm dn}\left( \frac{2 {\rm K}(k)}{L} x, k \right),
\end{equation}
where $k\in(0,1)$. The frequency $\omega$ depends smoothly on $k \in (0,1)$ and $L>0$ is defined by
\begin{equation}\label{omega-1}
	\omega = \frac{ 4(2-k^2) {\rm K}(k)^2}{L^2}.
\end{equation} 
Then, considering the transformation 
\begin{equation}\label{dnoidalsol}
	\varphi(x) = \frac{1}{(\kappa_1 + \gamma B^2)^{1/2}} \phi(x),
\end{equation} we obtain that $\varphi$ is an $L$-periodic solution for the equation \eqref{ode1} with $\omega$ given by the relation \eqref{omega-1} and defined in a subset of $(0,+\infty)$. We have the following result:

\begin{theorem}[Smooth Curve of Dnoidal Waves]\label{dnoidalcurve}
Let $L>0$ be fixed and consider  $\omega \in (\frac{2\pi^2}{L^2},+\infty)$. If $\varphi=\varphi_{\omega}$ is the solution of \eqref{ode1} with the dnoidal profile in \eqref{dnoidalsol}, the family
	$$
	\omega \in \left(\frac{2\pi^2}{L^2},+\infty\right) \longmapsto \varphi \in H^2_{per},
	$$
of periodic solutions of \eqref{ode1} depends smoothly on $\omega$.
\end{theorem}

 \subsection{$L$-periodic wave solutions with cnoidal profile}
 
 For $\omega> 0$, Angulo in \cite{angulo2007} obtained that the ODE
 \begin{equation*}
 	-\phi'' + \omega \phi - \phi^3 = 0
 \end{equation*} 
 also admits periodic solutions with cnoidal profile as
\begin{equation}\label{phi}
 \phi(x)=\frac{ \sqrt{2 \omega} k}{\sqrt{2k^2 -1}} {\rm cn}\left( \frac{4 {\rm K}(k)}{L} x, k \right),
 \end{equation}
 where $k\in\left(\frac{1}{\sqrt{2}},1\right)$. The frequency $\omega \in \mathbb{R}$ depends smoothly on $k \in \left(\frac{1}{\sqrt{2}},1\right)$ and $L>0$. It is defined by
 \begin{equation}\label{omega-2}
 	\omega = \frac{ 16 {\rm K}(k)^2 (2k^2 - 1)}{L^2}.
 \end{equation} 
 \indent Motivated by this, we consider the same scaling transformation as in \eqref{dnoidalsol} 
 \begin{equation}\label{cnoidalsol}
 	\varphi(x) = \frac{1}{(\kappa_1 + \gamma B^2)^{1/2}} \phi(x).
 \end{equation} We obtain that $\varphi$ is an $L$-periodic solution of \eqref{ode1} for $\omega >0$ given by \eqref{omega-2} and the similar result as in Theorem $\ref{dnoidalcurve}$ reads as follows:

 \begin{theorem}[Smooth Curve of Cnoidal Waves]\label{cnoidalcurve}
 	Let $L>0$ be fixed and consider  $\omega >0$. If $\varphi=\varphi_{\omega}$ is the solution of \eqref{ode1} with cnoidal profile given by \eqref{cnoidalsol}, then the family
 	$$
 	\omega \in (0,+\infty) \longmapsto \varphi=\varphi_{\omega} \in H^2_{per},
 	$$
 	of periodic solutions of \eqref{ode1} depends smoothly on $\omega \in (0,+\infty)$.
 \end{theorem}

\begin{remark}\label{cnoidal-odd}
	Recall that by $(\ref{cnoidalsol})$, we have
	\begin{equation}\label{cnoidal-2}
		\psi(x)=\varphi\left(x-\frac{L}{4}\right)=\frac{ \sqrt{2 \omega} k\sqrt{1-k^2}}{\sqrt{2k^2 - 1}} \frac{1}{(\kappa_1 + \gamma B^2)^{1/2}} \frac{{\rm sn}\left( \frac{ 4 {\rm K}(k)}{L} x, k \right)}{{\rm dn}\left( \frac{ 4 {\rm K}(k)}{L} x, k \right)}.
	\end{equation}
By Theorem $\ref{cnoidalcurve}$, we obtain that $$
\omega \in (0,+\infty) \longmapsto \psi=\psi_{\omega} \in H^2_{per,odd},
$$
depends smoothly on $\omega \in (0,+\infty)$.
\end{remark}

\section{Spectral analysis}\label{section4}

In this section, we calculate the non-positive spectrum of the linear operator $\mathcal{L}$ by using the information of the non-positive spectrum of the Hill operators $\mathcal{L}_i$, $i=1,2,3,4$, given by \eqref{Li}. For the case of dnoidal solutions, we borrow the results determined by Angulo \cite{angulo2007} and Hakkaev \cite{Hakkaev}. On the other hand, for the case of cnoidal solutions, we use some results given by Angulo in \cite{angulo2007} and Natali \textit{et. al.} in \cite{NataliMoraesLorenoPastor}. 

\subsection{Spectral analysis with dnoidal profile}

\indent Recall that operators $\mathcal{L}_i = -\partial_x^2 + \omega - \beta_i \varphi^2$ can be expressed in terms of the parameters $\beta_i$ in  \eqref{betai}. We have
\begin{equation*}
	\mathcal{L}_1 = -\partial_x^2 + \omega - \beta_1 \varphi^2 = -\partial_x^2 + \omega - 3(\kappa_1 + B^2 \gamma) \varphi^2 = -\partial_x^2 + \omega - 3 \phi^2
\end{equation*}
and
\begin{equation*}
	\mathcal{L}_2 = -\partial_x^2 + \omega - \beta_2 \varphi^2 = -\partial_x^2 + \omega - (\kappa_1 + B^2 \gamma) \varphi^2 = -\partial_x^2 + \omega -  \phi^2,
\end{equation*}
where $\phi$ is the solution with dnoidal profile given by \eqref{psi}. So, we can use the spectral properties for the operators $\mathcal{L}_1$ and $\mathcal{L}_2$ as obtained in \cite{angulo2007} and \cite{Hakkaev} to obtain
	${\rm n}(\mathcal{L}_1) = 1,\ {\rm n}(\mathcal{L}_2) = 0,$
	${\rm Ker}(\mathcal{L}_1) = [\varphi']$ and ${\rm Ker}(\mathcal{L}_2) = [\varphi].$
Thus, we need to study the spectral analysis of the operator $\mathcal{L}$ in some different cases:

\textbf{Case I:} $\gamma \in (0, \min\{\kappa_1,\kappa_2\})$. In this case, after some calculations with the parameters $\beta_i$, $i=1,2,3,4$, we conclude that
	\begin{equation}\label{ineq-1}
		\beta_4 < \beta_2 < \beta_3 < \beta_1.
	\end{equation}
for all $\gamma \in (0,\min\{\kappa_1,\kappa_2\})$. Then, we have the following order of operators
\begin{equation}\label{comparison}
	\mathcal{L}_1 < \mathcal{L}_3 < \mathcal{L}_2 < \mathcal{L}_4,
\end{equation} 
where $\mathcal{L}_i<\mathcal{L}_j$ means that $(\mathcal{L}_i u,u)_{L^2_{per}} < (\mathcal{L}_j u,u)_{L^2_{per}}$ for all $u \in H^2_{per}$, $u\neq0$ and $i,j\in\mathbb{N}$, $i\neq j$. From the comparison theorem (in the periodic context (see \cite[Theorem 2.2.2]{Eastham})) and the inequalities in \eqref{comparison}, we have 
	${\rm n}(\mathcal{L}_3) = 1,\ {\rm n}(\mathcal{L}_4) = 0$
and ${\rm z}(\mathcal{L}_3) = {\rm z}(\mathcal{L}_4) = \{0\}$. 

Therefore, it follows by $(\ref{ULU})$ that ${\rm n}(\mathcal{L}) = 2$ and ${\rm z}(\mathcal{L})=2$ with
$$
	{\rm Ker}(\mathcal{L}) = \left[(\varphi', B\varphi',0,0), (0,0,\varphi, B\varphi)\right].
$$

\textbf{Case II:} $\gamma \in (\max\{\kappa_1,\kappa_2\}, +\infty)$. By considering $\gamma \in (\max\{\kappa_1,\kappa_2\}, +\infty)$, one has
\begin{equation}\label{ineq-2}
	\beta_4 < \beta_3 < \beta_2 < \beta_1.
\end{equation}

So, we obtain by \eqref{ineq-2}
\begin{equation}\label{comparison2}
	\mathcal{L}_1 < \mathcal{L}_2 < \mathcal{L}_3 < \mathcal{L}_4.
\end{equation}
Thus, using the comparison theorem and \eqref{comparison2}, it follows that
${\rm n}(\mathcal{L}_3)={\rm n}(\mathcal{L}_4)=0$ and ${\rm z}(\mathcal{L}_3)={\rm z}(\mathcal{L}_4)=0$. 

Again, by $(\ref{ULU})$ we get ${\rm n}(\mathcal{L}) = 1$ and ${\rm z}(\mathcal{L})=2$ where
$$
	{\rm Ker}(\mathcal{L}) = \left[(\varphi', B\varphi',0,0), (0,0,\varphi, B\varphi)\right].
$$

\textbf{Case III:} $\gamma = 0$. First, we have to notice that for $B>0$, we obtain $B = \sqrt{\frac{\kappa_1}{\kappa_2}}$. By the expression of the matrix $S$ given by \eqref{S}, we do not need to use the similar transformation $R$. Operator $\mathcal{L}$ has a diagonal form and given by 
\begin{equation}\label{L-gamma0}
	\mathcal{L} = \left( \begin{array}{cccc}
	\mathcal{L}_1 & 0 & 0 & 0 \\
	0 & \mathcal{L}_1 & 0 & 0 \\
	0 & 0 & \mathcal{L}_2 & 0 \\
	0 & 0 & 0 & \mathcal{L}_2
	\end{array}
\right),
\end{equation}
that is, $\mathcal{L}_3=\mathcal{L}_1$ and $\mathcal{L}_4=\mathcal{L}_2$.\\
\indent The spectral analysis can be directly determined by the behaviour of the linearized operators $\mathcal{L}_i$, $i=1,2,3,4$. Since ${\rm n}(\mathcal{L}_1) = 1,\ {\rm n}(\mathcal{L}_2) = 0$ and
	${\rm z}(\mathcal{L}_1)={\rm z}(\mathcal{L}_2) = 1,$  we have that ${\rm n}(\mathcal{L}) = 2$ and ${\rm z}(\mathcal{L}) = 4$ where
$$
	{\rm Ker}(\mathcal{L}) = \left[ (\varphi', 0,0,0), (0,\varphi', 0,0), (0,0,\varphi,0), (0,0,0,\varphi) \right].
$$

\textbf{Case IV:} $\gamma = \kappa_1 = \kappa_2$. In this case, we have that $B$ is a free real parameter. Moreover, we also use the similar transformation $S = R M R^{-1}$, where $S$ is given by \eqref{S}. The matrix $M$ and $R$ are then given respectively by
\begin{equation}\label{M-Bfree}
	M = \left( \begin{array}{cccc}
		3(B^2+1)\gamma & 0 & 0 & 0 \\
		0 & (B^2 + 1) \gamma & 0 & 0 \\
		0 & 0 & (B^2 + 1) \gamma & 0 \\
		0 & 0 & 0 & -(B^2 + 1)\gamma
	\end{array}
\right),
\end{equation}
and
\begin{equation}\label{U-Bfree}
	R = \left( \begin{array}{cccc}
		-\frac{1}{B^2+1} & -B & 0 & 0 \\
		-\frac{B}{B^2+1} & 1 & 0 & 0 \\
		0 & 0 & 1 & \frac{B}{B^2+1}  \\
		0 & 0 & B & -\frac{1}{B^2+1}
	\end{array}
\right).
\end{equation}
Thus, operator $\mathcal{L}$ becomes in this case
\begin{equation}\label{L-Bfree}
	\mathcal{L} = R \left( \begin{array}{cccc}
	\mathcal{L}_1 & 0 & 0 & 0 \\
	0 & \mathcal{L}_2 & 0 & 0 \\
	0 & 0 & \mathcal{L}_2 & 0 \\
	0 & 0 & 0 & \mathcal{L}_3
	\end{array}
\right) R^{-1},
\end{equation}
where
\begin{equation}\label{L1L2L3-Bfree}
	\mathcal{L}_1 = -\partial_x^2 + \omega - 3(B^2+1)\gamma \varphi^2, \; \mathcal{L}_2 = -\partial_x^2 + \omega - (B^2+1) \gamma \varphi^2 \text{ and } \mathcal{L}_3 = -\partial_x^2 + \omega + (B^2+1) \gamma \varphi^2.
\end{equation}

As far as we can see, we have that $(B^2+1)\gamma = \kappa_1 + B^2 \gamma$. So, we obtain the same spectral properties concerning the operators $\mathcal{L}_1$ and $\mathcal{L}_2$. In addition, being $\mathcal{L}_3$ a positive operator, we have that ${\rm n}(\mathcal{L}) = 1$ and ${\rm z}(\mathcal{L}) = 3$, where 
$$
	{\rm Ker}(\mathcal{L}) = \left[ (\varphi', B\varphi',0,0), (-B\varphi, \varphi,0,0), (0, 0, \varphi, B\varphi) \right].
$$
\subsection{Spectral analysis with cnoidal profile}\label{cnoidalprof}

Let $\varphi$ be the solution with cnoidal profile given by Theorem \ref{cnoidalcurve}. For $B>0$ given by \eqref{B}, the transformation $R$ such that $S =R M R^{-1}$ is also given by \eqref{U-Bfixed}. As a consequence, the parameters $\beta_i$, $i=1,2,3,4$ are also given by \eqref{betai} and operators $\mathcal{L}_1$ and $\mathcal{L}_2$ can be expressed as
\begin{equation}\label{eq-L2}
	\begin{array}{c}
	\mathcal{L}_1 = -\partial_x^2 + \omega - \beta_1 \varphi^2 = -\partial_x^2 + \omega - 3(\kappa_1 + B^2 \gamma) \varphi^2 = -\partial_x^2 + \omega - 3 \phi^2 \text{ and } \vspace{0.3cm} \\
	\mathcal{L}_2 = -\partial_x^2 + \omega - \beta_2 \varphi^2 = -\partial_x^2 + \omega - (\kappa_1 + B^2 \gamma) \varphi^2 = -\partial_x^2 + \omega -  \phi^2,
	\end{array}
\end{equation}
where $\phi$ is the solution with cnoidal profile in \eqref{phi}. So, we can use the spectral properties for the operators $\mathcal{L}_1$ and $\mathcal{L}_2$ as obtained in \cite{angulo2007}. Indeed, we have
${\rm n}(\mathcal{L}_1) = 2,\ {\rm n}(\mathcal{L}_2) = 1$ and ${\rm z}(\mathcal{L}_1)={\rm z}(\mathcal{L}_2)=1$. 

 It is necessary to understand that when the number of negative eigenvalues is high (compared with the case of dnoidal solutions), we obtain some difficulties to obtain the spectral properties for the operator $\mathcal{L}$ in $(\ref{L-1})$. So, we can study some different cases:

\textbf{Case I:} $\gamma \in (0,\min\{\kappa_1,\kappa_2\})$. By \eqref{ineq-1}, we have that 
\begin{equation*}
	\mathcal{L}_1 < \mathcal{L}_3 < \mathcal{L}_2 < \mathcal{L}_4.
\end{equation*}
Then, using the comparison theorem one has ${\rm n}(\mathcal{L}_3) = 2$ and ${\rm z}(\mathcal{L}_3)=0$. In addition, we have that $\beta_4 < 0$ for all $\gamma \in (\frac{1}{3} ( \kappa_1+\kappa_2-\sqrt{\kappa_1^2-\kappa_1\kappa_2+\kappa_2^2}), \min\{\kappa_1,\kappa_2\})$ and this implies that ${\rm n}(\mathcal{L}_4) = {\rm z}(\mathcal{L}_4) = 0$. Therefore, we have in this case
${\rm n}(\mathcal{L}) = 5 \text{ and } {\rm z}(\mathcal{L}) = 2$ with
$${\rm Ker}(\mathcal{L}) = \left[ (\varphi', B\varphi', 0,0), (0,0,\varphi, B\varphi) \right].$$
\textbf{Case II:} $\gamma \in (\max\{\kappa_1,\kappa_2\},+\infty)$. By \eqref{ineq-2}, we obtain
\begin{equation*}
	\mathcal{L}_1 < \mathcal{L}_2 < \mathcal{L}_3 < \mathcal{L}_4.
\end{equation*}
Using the comparison theorem and the previous knowledge of the non-positive spectrum for $\mathcal{L}_1$ and $\mathcal{L}_2$, we cannot determine the behaviour of the non-positive spectrum for $\mathcal{L}_3$. In fact, we can obtain the following scenarios:
\begin{equation}\label{tricotomy}
	{\rm n}(\mathcal{L}_3) = 1 \text{ and } {\rm z}(\mathcal{L}_3) = 0, \; \; \; \;	{\rm n}(\mathcal{L}_3) = 0 \text{ and } {\rm z}(\mathcal{L}_3) = 1 \; \; \text{ or } \; \; 	{\rm n}(\mathcal{L}_3) = {\rm z}(\mathcal{L}_3) = 0.
\end{equation}
Thus, the spectral analysis becomes inconclusive and we cannot obtain the required spectral stability for the cnoidal waves in this case.

\textbf{Case III:} $\gamma = 0$. As in the dnoidal case, we have $B = \sqrt{\frac{\kappa_1}{\kappa_2}}$ and consequently, the operator $\mathcal{L}$ is a diagonal operator as in \eqref{L-gamma0}. Thus, from the comparison theorem we obtain 
	\begin{equation*}
		{\rm n}(\mathcal{L}) = 6 \text{ and } {\rm Ker}(\mathcal{L}) = \left[ (\varphi',0,0,0), (0,\varphi',0,0), (0,0,\varphi,0), (0,0,0,\varphi) \right].
	\end{equation*}

\textbf{Case IV:} $\gamma = \kappa_1 = \kappa_2$. Here $B \in \mathbb{R}$ is a free parameter and we obtain, as in case of dnoidal solutions, the similar transformation $S = RMR^{-1}$, where $M$ and $R$ are given by \eqref{M-Bfree} and \eqref{U-Bfree}, respectively. Thus, the operator $\mathcal{L}$ can be also expressed by \eqref{L-Bfree} where the operators $\mathcal{L}_i$, $i=1,2,3,4$ are given by \eqref{L1L2L3-Bfree}. 

We have to notice that $(B^2 + 1)\gamma = \kappa_1 + B^2 \gamma$, ${\rm n}(\mathcal{L}_1) = 2$ and ${\rm n}(\mathcal{L}_2) = 1$ and $\mathcal{L}_3$ being a positive operator. Summarizing all mentioned results, we get
\begin{equation*}
	{\rm n}(\mathcal{L}) = 4 \text{ and } {\rm Ker}(\mathcal{L}) = \left[ (\varphi', B\varphi', 0,0), (-B\varphi, \varphi,0,0), (0,0,\varphi, B \varphi)\right]. 
\end{equation*}

\subsection{Spectral analysis in $H_{per,odd}^2$}\label{subsection-odd} To address the excessive number of negative eigenvalues, we must impose a suitable restriction on the operator $\mathcal{L}$ in $(\ref{L-1})$. First, by the spectral analysis for the cnoidal solution in \cite{angulo2007}, we obtain that the first three eigenvalues and the corresponding eigenfunctions of $\mathcal{L}_1$ are 
\begin{align}
\lambda_0 = (1 - 6k^2 - 2a(k))\left(\frac{16 {\rm K}(k)^2}{L^2} \right) \; \; \; & \; \; \; \phi_0(x) = k^2 {\rm sn}^2\left( \frac{4 {\rm K}(k)}{L} x,k\right) - \frac{1}{3} (1+k^2+a(k)), \label{eigf-1} \\
\lambda_1 = -3k^2 \left(\frac{16 {\rm K}(k)^2}{L^2} \right) \; \; \; & \; \; \; \phi_1(x) =  {\rm cn}\left( \frac{4 {\rm K}(k)}{L} x,k\right) {\rm dn}\left( \frac{4 {\rm K}(k)}{L} x,k\right), \\
\lambda_2 = 0 \; \; \; & \; \; \; \phi_2(x) = \partial_x {\rm cn}\left( \frac{4 {\rm K}(k)}{L} x,k\right),
\end{align}
where $a(k):=\sqrt{1-k^2+k^4}$. In addition, the first two eigenvalues and the corresponding eigenfunctions for the linearized operator $\mathcal{L}_2$ are 
\begin{align}
	\lambda_0 = \frac{16 {\rm K}(k)^2 (k^2-1)}{L^2} \; \; \; & \; \; \; \phi_0(x) = {\rm dn}\left( \frac{4 {\rm K}(k)}{L} x,k\right), \\
	\lambda_1 = 0 \; \; \; & \; \; \; \phi_1(x) = {\rm cn}\left( \frac{4 {\rm K}(k)}{L} x,k\right). \label{eigf-2}
\end{align}

Considering $\psi$ as in \eqref{cnoidal-2}, we have that the operators
 $$\mathcal{L}_{i,odd}:= \mathcal{L}_i: H^2_{per,odd} \rightarrow L^2_{per,odd},$$
 are well defined for all $i=1,2,3,4$. Thus, by applying the transformation $f=g(\cdot-L/4)$ in all eigenfunctions given in \eqref{eigf-1}-\eqref{eigf-2}, we conclude
\begin{equation}\label{L1L2-odd}
	{\rm n}(\mathcal{L}_{1,odd}) = 1 \; \; \text{ and } {\rm n}(\mathcal{L}_{2,odd}) = 0.
\end{equation}
In addition,
\begin{equation*}
	{\rm Ker}(\mathcal{L}_{1,odd}) = \{0\} \; \; \text{ and } \; \; {\rm Ker}(\mathcal{L}_{2,odd}) = [\psi].
\end{equation*}

Therefore, since we have almost the same scenario as determined for the case of dnoidal solutions, we can use the comparison theorem without further problems. We can analyse again the cases in this new perspective:

\textbf{Case I:} $\gamma \in (0,\min\{\kappa_1,\kappa_2\})$. Since in this case, we have $\mathcal{L}_{1,odd} < \mathcal{L}_{3,odd} < \mathcal{L}_{2,odd} < \mathcal{L}_{4,odd}$, we obtain by the comparison theorem and \eqref{L1L2-odd} that ${\rm n}(\mathcal{L}_{3,odd}) = 1$, ${\rm n}(\mathcal{L}_{4,odd}) = 0$ and ${\rm Ker}(\mathcal{L}_{3,odd}) = {\rm Ker}(\mathcal{L}_{4,odd}) = \{0\}$. Therefore,
\begin{equation*}
	{\rm n}(\mathcal{L}_{odd}) = 2 \text{ and } {\rm Ker}(\mathcal{L}_{odd}) = \left[ (0,0,\psi,B\psi)\right].
\end{equation*}

\textbf{Case II:} $\gamma \in (\max\{\kappa_1,\kappa_2\},+\infty)$. Here, we have $\mathcal{L}_{1,odd} < \mathcal{L}_{2,odd} < \mathcal{L}_{3,odd} < \mathcal{L}_{4,odd}$. By the comparison theorem and \eqref{L1L2-odd}, we obtain ${\rm n}(\mathcal{L}_{3,odd}) = {\rm n}(\mathcal{L}_{4,odd}) = 0$ and ${\rm Ker}(\mathcal{L}_{3,odd}) = {\rm Ker}(\mathcal{L}_{4,odd}) = \{0\}$. Thus, 
\begin{equation*}
	{\rm n}(\mathcal{L}_{odd}) = 1 \text{ and } {\rm Ker}(\mathcal{L}_{odd}) = \left[ (0,0,\psi,B\psi)\right].
\end{equation*}

\textbf{Case III}: $\gamma = 0$. In this case, we have that $B = \sqrt{\frac{\kappa_1}{\kappa_2}}$ and 
\begin{equation*}
	\mathcal{L}_{odd} = \left(
	\begin{array}{cccc}
		\mathcal{L}_{1,odd} & 0 & 0 & 0 \\
		0 & \mathcal{L}_{1,odd} & 0 & 0 \\
		0 & 0 & \mathcal{L}_{2,odd} & 0 \\
		0 & 0 & 0 & \mathcal{L}_{2,odd}
	\end{array}
\right).
\end{equation*}
So, we get
\begin{equation*}
	{\rm n}(\mathcal{L}_{odd}) = 2 \text{ and } {\rm Ker}(\mathcal{L}_{odd}) = \left[ (0,0,\psi,0), (0,0,0,\psi) \right].
\end{equation*}

\textbf{Case IV:} $\gamma = \kappa_1 = \kappa_2$. In this case, $B \in \mathbb{R}$ is a free parameter and we can use the similar transformation $S = RMR^{-1}$ where $M$ and $R$ given in \eqref{M-Bfree} and \eqref{U-Bfree}, respectively (similar to the case of dnoidal waves). Thus,  since $\mathcal{L}_{3,odd}$ is now positive, the operator $\mathcal{L}$ in \eqref{L-Bfree} restricted to $H_{per,odd}^2$ satisfies
\begin{equation*}
	{\rm n}(\mathcal{L}_{odd})=1 \text{ and } {\rm Ker}(\mathcal{L}_{odd}) = \left[ (-B\psi, \psi,0,0), (0,0,\psi, B\psi) \right].
\end{equation*}

\section{Spectral stability}\label{section5}
In this section, we obtain spectral stability results for the periodic multiple solution $\Phi = (\varphi, B\varphi,0,0)$ considering three different scenarios:  when $\varphi$ has a dnoidal profile, when $\varphi$ has a cnoidal profile, and when $\psi=\varphi(\cdot-L/4)$ has a snoidal profile and it is restricted to the space of odd functions $H^1_{per,odd}$. To do so, we need to obtain the entries of the matrix $V$ in \eqref{V}. In fact, we can obtain a simplified way to obtain the matrix $V$ using the transformation $(\ref{ULU})$ and the fact that $\Theta_l\in \Ker(\mathcal{L})$. Thus, we have
\begin{equation*}
V_{jl} = ( \mathcal{L}^{-1} J \Theta_j, J \Theta_l )_{\mathbb{L}^2_{per}} = ( \tilde{\mathcal{L}}^{-1} R^{-1} J \Theta_j, R^{-1} J \Theta_l )_{\mathbb{L}^2_{per}}.
\end{equation*}

To determine our spectral stability result, we also consider the spectral analysis of the operator $\mathcal{L}$ in $(\ref{L-1})$ as determined in the last section. Before presenting all possible cases of the matrix $V$, we need to introduce two important and well-known facts (a remark and a lemma). Both of them are useful to improve the reader's understanding.

\begin{remark}\label{positivity}
Let  $\mathcal{A}$ be a self-adjoint operator defined in a Hilbert space $H$ with dense domain $D(\mathcal{A})$. Suppose also that its spectrum $\sigma(\mathcal{A})$ is constituted only by an infinite discrete set of eigenvalues and satisfying $\sigma(\mathcal{A})\subset[0,+\infty)$. There exists $\delta > 0$ such that 
\begin{equation*}
	(\mathcal{A} v, v)_{H} \geq \delta \|v\|_{H}^2
\end{equation*}
for all $v \in D(\mathcal{A})$ satisfying $v \in {\rm Ker}(\mathcal{A})^\perp$. In fact,
	since $H$ is a Hilbert space, we have the decomposition $H = {\rm Ker}(\mathcal{A}) \oplus {\rm Ker}(\mathcal{A})^\perp$. From Theorem 6.17 in \cite[page 178]{Kato}, we have
	\begin{equation*}
		\sigma( \mathcal{A} ) = \sigma \left( \mathcal{A} \big{|}_{{\rm Ker}(\mathcal{A})} \right) \cup \sigma\left( \mathcal{A} \big{|}_{{\rm Ker}(\mathcal{A})^\perp} \right).
	\end{equation*}
	On the other hand, we have that
	\begin{equation*}
	\sigma\left( \mathcal{A} \big{|}_{{\rm Ker}(\mathcal{A})^\perp} \right) = \sigma(\mathcal{A}) \setminus \{0\},
	\end{equation*}
	that is, the spectrum is bounded from below. The arguments in \cite[page 279]{Kato} imply that $\mathcal{A}$ is also bounded from below. Therefore, there exists $\delta > 0$ satisfying
	\begin{equation*}
		(\mathcal{A} v, v) \geq \delta \|v\|_{L^2_{per}}^2 \text{ for all } v \in H^2_{per} \cap {\rm Ker}(\mathcal{A})^\perp.
	\end{equation*}
\end{remark}

\begin{lemma}\label{dif-norma}
	Let $L>0$ be fixed. Consider the smooth periodic waves $\varphi$ with dnoidal and cnoidal profiles given by Theorem \ref{dnoidalcurve} and Theorem \ref{cnoidalcurve}, respectively. Then, it follows that $\frac{d}{d\omega} \|\varphi\|_{L_{per}^2}^2 > 0$.
\end{lemma}
\begin{proof} 
	Let $L>0$ be fixed and consider $\varphi$ as the dnoidal profile given by Theorem $\ref{dnoidalcurve}$. Then, by \cite[Formula 314.02]{byrd} we get
	\begin{align*}
		\int_{0}^{L} \varphi(x)^2 dx & = \frac{ 8{\rm K}(k)^2}{L^2(\kappa_1 + \gamma B^2)} \int_{0}^L {\rm dn}^2 \left( \frac{2 {\rm K}(k)}{L} x,k \right) dx \\
		& = \frac{ 8 {\rm K}(k)}{L (\kappa_1 + \gamma B^2)} \int_{0}^{{\rm K}(k)} {\rm dn}^2 (u,k) du \\
		& = \frac{8}{L(\kappa_1 + \gamma B^2)} {\rm E}(k) {\rm K}(k).
	\end{align*} 
Thus, we obtain 
$$
	\frac{d}{d\omega} \|\varphi\|_{L^2_{per}}^2 = \frac{d}{d\omega} \left( \frac{8}{L(\kappa_1 + \gamma B^2)} {\rm E}(k) {\rm K}(k) \right) = \frac{8}{L(\kappa_1 + \gamma B^2)} \frac{d}{dk} \left( {\rm E}(k) {\rm K}(k) \right) \left( \frac{d\omega}{d k} \right)^{-1} > 0
$$
for all $k \in (0,1)$.\\
\indent On the other hand, let $\varphi$ be the cnoidal profile given by Theorem 
	$\ref{cnoidalcurve}$. By \cite[Formula 312.02]{byrd}, we have
	\begin{align*}
		\int_{0}^{L} \varphi(x)^2 dx & = \frac{ 32 k^2 {\rm K}(k)^2}{L^2(\kappa_1 + \gamma B^2)} \int_{0}^L {\rm cn}^2 \left( \frac{4 {\rm K}(k)}{L} x,k \right) dx \\
		& = \frac{ 32 k^2 {\rm K}(k)}{L (\kappa_1 + \gamma B^2)} \int_{0}^{{\rm K}(k)} {\rm cn}^2 (u,k) du \\
		& = \frac{32}{L(\kappa_1 + \gamma B^2)} \left( {\rm K}(k) \left( {\rm E}(k) - (1-k^2) {\rm K}(k) \right) \right).
	\end{align*}
	Thus, we obtain that
	$$
		\frac{d}{d\omega} \|\varphi\|_{L^2_{per}}^2 = \frac{32}{L (\kappa_1 + \gamma B^2)} \frac{d}{dk} \left[ {\rm K}(k) \left( {\rm E}(k) - (1-k^2) {\rm K}(k) \right) \right] \left( \frac{ d\omega}{dk} \right)^{-1} > 0
	$$
	for all $k \in (0,1)$.
\end{proof}

\begin{remark}
	The result obtained in Lemma \ref{dif-norma} can be applied in the case of odd periodic waves $\psi$ determined in Remark $\ref{cnoidal-odd}$. 
\end{remark}

\subsection{Spectral stability for the multiple periodic wave solution with dnoidal profile}
In what follows, we consider the dnoidal wave solution $\varphi$ determined by Theorem \ref{dnoidalcurve}. Since we have separated the spectral analysis into four cases, we need to consider the same four cases in order to establish the spectral stability for the multiple solution $\Phi=(\varphi,B\varphi,0,0)$.

\textbf{Case I:} $\gamma \in (0,\min\{\kappa_1,\kappa_2\})$. Since ${\rm z}(\mathcal{L})=2$ with $\Theta_1=(\varphi',B\varphi',0,0)$ and $\Theta_2=(0,0,\varphi,B\varphi)$, we obtain that the matrix $V$ is $2\times 2$ and given by
\begin{equation}\label{V-1}
\begin{array}{lllll}
V &=& \left( \begin{array}{cc}
		(\tilde{\mathcal{L}}^{-1} R^{-1} J \Theta_1, R^{-1} J \Theta_1)_{L_{per}^2} & (\tilde{\mathcal{L}}^{-1} R^{-1} J \Theta_1, R^{-1} J \Theta_2)_{L_{per}^2} \\
		(\tilde{\mathcal{L}}^{-1} R^{-1} J \Theta_2, R^{-1} J \Theta_1)_{L_{per}^2} & (\tilde{\mathcal{L}}^{-1} R^{-1} J \Theta_2, R^{-1} J \Theta_2)_{L_{per}^2}
	\end{array}
\right).\\\\
	&=&\left( \begin{array}{cc}
	(2\gamma - \kappa_1 - \kappa_2)^2 ( \mathcal{L}_2^{-1} \varphi', \varphi' )_{L^2_{per}} & 0 \\
	0 & (\gamma - \kappa_2)^{-2} ( \mathcal{L}_1^{-1} \varphi,\varphi )_{L^2_{per}}
\end{array}
	\right),
\end{array}
\end{equation}
where we are using the similar transformation $\mathcal{L}=R\tilde{\mathcal{L}}R^{-1}$ to obtain a more convenient expression for the entries of the matrix $V$.\\
\indent Since $\varphi' \in ({\rm Ker}(\mathcal{L}_2))^{\perp} = {\rm Range}(\mathcal{L}_2)$, there exists $\xi \in D(\mathcal{L}_2)$ such that $\mathcal{L}_2 \xi = \varphi'$. Since $\mathcal{L}_2$ does not have negative eigenvalues, we obtain that $\xi$ satisfy the conditions of Remark \ref{positivity}, so that $( \mathcal{L}_2^{-1} \varphi',\varphi' )_{L^2_{per}} > 0$. On the other hand, by Theorem $\ref{dnoidalcurve}$, we can derive the equation \eqref{ODE} with respect to $\omega$ to obtain that $\mathcal{L}_1 \left( \frac{d}{d\omega} \right) = -\varphi$. Thus, by Lemma \ref{dif-norma} we get
\begin{equation*}
	( \mathcal{L}_1^{-1} \varphi, \varphi )_{L^2_{per}} = - \left( \frac{d}{d\omega} \varphi, \varphi \right)_{L^2_{per}} = - \frac{1}{2} \frac{d}{d\omega} \|\varphi\|_{L^2_{per}}^2 < 0.
\end{equation*}
Thus, we have ${\rm n}(V) = 1$. Since ${\rm n}(\mathcal{L}) = 2$, we conclude that the multiple solution $\Phi = (\varphi, B\varphi,0,0)$ is spectrally unstable. 

\textbf{Case II:} $\gamma \in (\max\{\kappa_1,\kappa_2\},+\infty)$. The kernel of $\mathcal{L}$ in this case has the same elements as in the last case, so that the matrix $V$ is the same as in \eqref{V-1}. Since we have ${\rm n}(\mathcal{L})={\rm n}(V) = 1$, we conclude that the multiple solution $\Phi = (\varphi, B \varphi,0,0)$ is spectrally stable. 

\textbf{Case III:} $\gamma = 0$. Since ${\rm z}(\mathcal{L})=4$ with $\Theta_1=(\varphi', 0,0,0)$, $\Theta_2=(0,\varphi', 0,0)$, $\Theta_3=(0,0,\varphi,0)$, and $\Theta_4=(0,0,0,\varphi)$, we obtain that the matrix $V$ is $4\times 4$ and given by
\begin{equation*}
	V = \left( \begin{array}{cccc}
		( \mathcal{L}_2^{-1} \varphi', \varphi' )_{L^2_{per}} & 0 & 0 & 0 \\
		0 & ( \mathcal{L}_2^{-1} \varphi',\varphi' )_{L^2_{per}} & 0 & 0 \\
		0 & 0 & ( \mathcal{L}_1^{-1} \varphi, \varphi )_{L^2_{per}} & 0 \\
		0 & 0 & 0 & ( \mathcal{L}_1^{-1} \varphi, \varphi )_{L^2_{per}}
	\end{array}
\right).
\end{equation*}
Doing the same calculations as we have already performed in the first case, we obtain ${\rm n}(V) = 2$. Since ${\rm n}(\mathcal{L}) = 2$, we have that the difference ${\rm n}(\mathcal{L}) - {\rm n}(V)$ is zero and the periodic multiple solution is spectrally stable.\\
\indent  \textbf{Case IV:} $\gamma = \kappa_1 = \kappa_2$. Now, we have ${\rm z}(\mathcal{L})=3$ with 
$\Theta_1=(\varphi', B\varphi',0,0)$, $\Theta_2=(-B\varphi, \varphi,0,0)$ and 
$\Theta_3=(0, 0, \varphi, B\varphi)$. We obtain that the matrix $V$ is $3\times 3$ and given by
\begin{equation*}
	V = \left( \begin{array}{ccc}
		( \mathcal{L}_2^{-1} \varphi',\varphi' )_{L^2_{per}} & 0 & 0 \\
		0 & (B^2 + 1)^2 ( \mathcal{L}_3^{-1} \varphi,\varphi )_{L^2_{per}} & 0 \\
		0 & 0 & (B^2+1)^2 ( \mathcal{L}_1^{-1} \varphi, \varphi )_{L^2_{per}}
	\end{array}
\right).
\end{equation*}
We see in this case that $\mathcal{L}_3$ is positive, so that $( \mathcal{L}_3^{-1} \varphi,\varphi )_{L^2_{per}}>0$. By similar arguments as determined in the first case, we then obtain ${\rm n}(V) = 1$. Since ${\rm n}(\mathcal{L}) = 1$, we get that the multiple solution $\Phi = (\varphi, B \varphi,0,0)$ is spectrally stable. 

Summarizing the above, we have proved Theorem \ref{teo-1}.

\begin{remark}\label{remstab}
	
	The abstract theories in \cite{grillakis1} and \cite{grillakis2} can be used to establish the orbital stability of periodic dnoidal waves in certain cases, where we have previously established the spectral stability as determined in this subsection. To this end, we need to have the following set of conditions:\\
	\begin{itemize}
		\item ${\rm n}(\mathcal{L})=1$,
		\item ${\rm z}(\mathcal{L})=2$,
		\item ${\rm n}(V)=1$ with $( \mathcal{L}_1^{-1} \varphi, \varphi )_{L^2_{per}}<0$.
	\end{itemize}
The three requirements mentioned above occur exactly in the second case above (Case II). In the first case (Case I), the orbital instability in the space $\mathbb{H}_{per,even}^1$ constituted by even periodic functions in $\mathbb{H}_{per}^1$ can be established using the instability results in \cite{grillakis1} and \cite{grillakis2}.
\end{remark}

\subsection{Spectral stability for the multiple periodic wave solution with cnoidal profile}
Here, we consider the periodic multiple wave solution $\Phi = (\varphi, B \varphi,0,0)$ where $\varphi$ has the cnoidal profile and we determine the corresponding spectral stability. After that, we consider the translation solution $\psi=\varphi(\cdot-L/4)$ to study the spectral stability in  $\mathbb{H}_{per,odd}^1$ for the case where the operator $\mathcal{L}$ has too many negative eigenvalues.


We have to notice that in the case of dnoidal profile, we have used Remark \ref{positivity} to obtain that $( \mathcal{L}_2^{-1} \varphi',\varphi' )_{L^2_{per}} > 0$. To do so, we need to use the fact that ${\rm n}(\mathcal{L}_2) = 0$. However, when we are considering the cnoidal profile, we obtain ${\rm n}(\mathcal{L}_2) = 1$ and this property does not allow us to use directly Remark \ref{positivity} to evaluate the positiveness of $( \mathcal{L}_2^{-1} \varphi',\varphi')_{L^2_{per}}$. This difficulty can be avoided by the following lemma:

\begin{lemma}\label{LemmaL2-cn}
	Let $L>0$ be fixed and consider $\varphi$ the periodic wave with cnoidal profile given by Theorem $\ref{cnoidalcurve}$. We have that
	\begin{equation*}
		( \mathcal{L}_2^{-1} \varphi',\varphi' )_{L^2_{per}} > 0.
	\end{equation*}
\end{lemma}
\begin{proof}
By \eqref{eq-L2}, we can rewrite $\mathcal{L}_2 = -\partial_x^2 + \omega - \phi^2$ where $\phi$ is the cnoidal wave solution in \eqref{phi}. Since $\varphi$ in \eqref{cnoidalsol} is a multiple of $\phi$, there exists $\chi \in D(\mathcal{L}_2)$ such that $\mathcal{L}_2 \chi = \phi'$. To calculate the value of $( \mathcal{L}_2^{-1} \varphi',\varphi' )_{L^2_{per}}$, it suffices to evaluate the quantity $(\chi,\phi')_{L^2_{per}}$. To this end, we use a similar approach as in \cite[Section 3]{NataliCardosoAmaral} (see also \cite{BittencourtLoreno} and \cite{NataliMoraesLorenoPastor}).

We can start by noticing that $\lambda = 0$ is a simple eigenvalue with associated eigenfunction $\phi$. Thus, there exists a smooth non-periodic function $y$ satisfying the Hill equation
\begin{equation}\label{y-hillequation}
	- y'' + \omega y + \phi^2 y = 0,
\end{equation}
and $\{\varphi, y\}$ is the fundamental set of solutions for the equation $\eqref{y-hillequation}$. Since $\phi$ is even, we have that $y$ is odd and it satisfies the following
\begin{equation*}
	\left\{ \begin{array}{l}
		-y'' + \omega y - \phi^2 y = 0 \\
		y(0) = 0 \\
		y'(0) = \frac{1}{\phi(0)}.
		\end{array}
	\right.
\end{equation*}
\indent Next, we see that $\chi \in D(\mathcal{L}_2)$ satisfies the equation $\mathcal{L}_2 \chi = \phi'$, so that
\begin{equation}\label{chi-hill}
	-\chi'' + \omega \chi - \phi^2 \chi = \phi'.
\end{equation}
Multiplying \eqref{chi-hill} by ${y}$, integrating over $[0,L]$ and using integration by parts, we obtain 
\begin{equation*}
	\chi'(0) = - \frac{ \int_{0}^L \phi'(x) {y(x)} dx}{{y}(L)},
	\end{equation*}
where we are using the fact that $y$ is not periodic, so that $y(L)\neq0$. The fact that $\chi$ is an odd function gives us the following IVP:
\begin{equation}\label{chi-pvi}
	\left\{ \begin{array}{l}
		-\chi'' + \omega \chi - \phi^2 \chi = \phi' \\
		\chi(0) = 0 \\
		\chi'(0) = -\frac{1}{{y}(L)} \int_{0}^L \phi'(x) {y(x)} dx.
	\end{array}
\right.
\end{equation}
Problem \eqref{chi-pvi} is suitable to perform some numeric calculations. In fact, we can deduce that 
\begin{equation*}
(\chi, \phi')_{L^2_{per}} = \frac{\eta(k)}{L}
\end{equation*}
where $\eta$ is a positive constant depending only on $k \in (\frac{1}{\sqrt{2}},1)$. We obtain $(\chi, \phi')_{L^2_{per}} > 0$ for all $k \in (\frac{1}{\sqrt{2}},1)$ (see Figure $\ref{Figura}$), so that
\begin{equation*}
	(\mathcal{L}_2^{-1} \varphi',\varphi')_{L^2_{per}} = \frac{1}{(\kappa_1 + \gamma B^2)} (\mathcal{L}_2^{-1} \phi',\phi')_{L^2_{per}} = (\chi, \phi')_{L^2_{per}} > 0.
\end{equation*}



\begin{figure}[!h]
	\includegraphics[width=5.5cm,height=6cm]{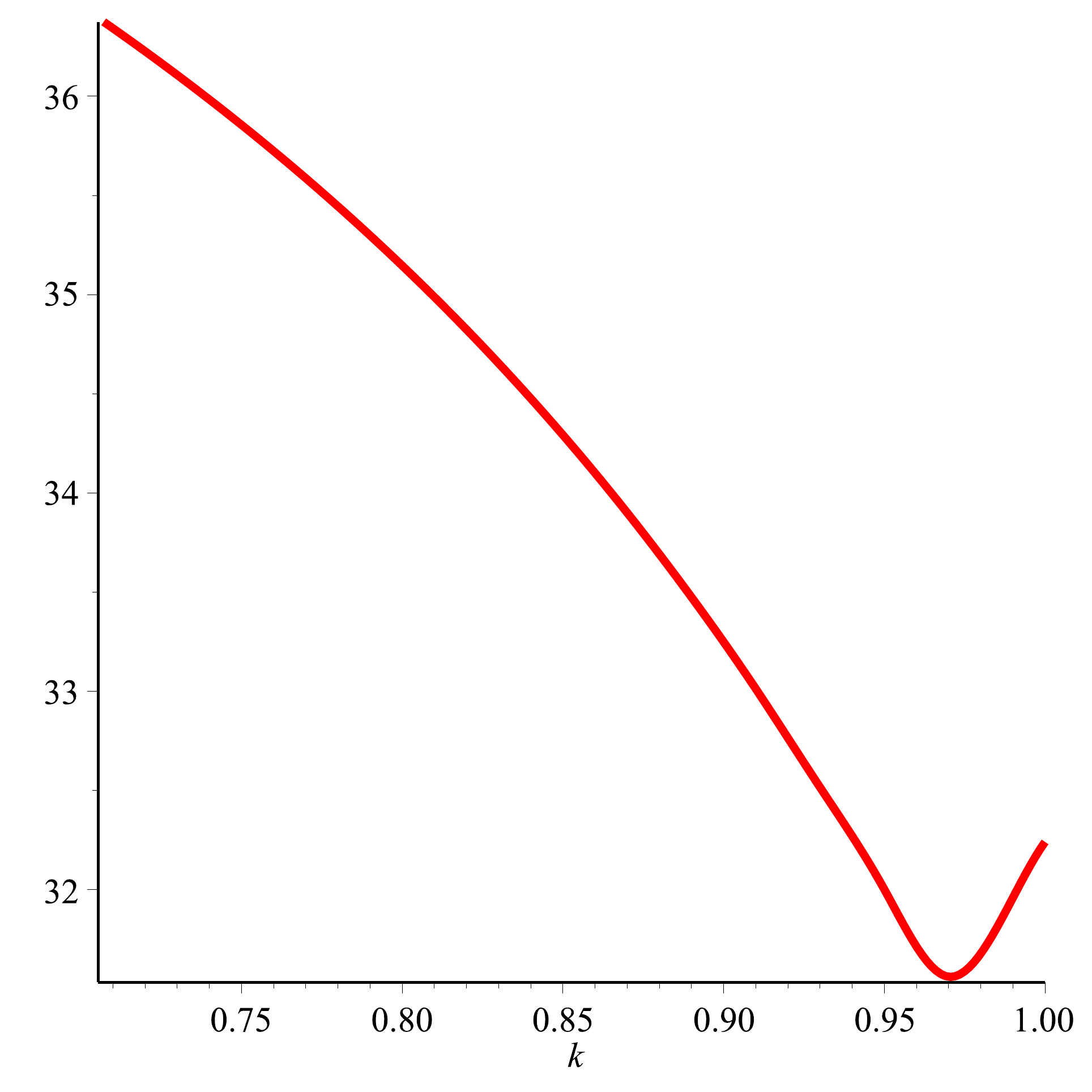}
	\caption{Behaviour of the quantity $\eta(k)$ with $k \in (\frac{1}{\sqrt{2}}, 1)$.}
	\label{Figura}
\end{figure}
\end{proof}

Results above allow us to determine the spectral stability of the multiple periodic wave $\Phi = (\varphi, B\varphi,0,0)$ with cnoidal profile. The analysis is quite similar as determined in the last subsection, and so we only give the main steps.

\textbf{Case I:} $\gamma \in (0,\min\{\kappa_1,\kappa_2\})$. Since ${\rm z}(\mathcal{L})=2$, we have that $V$ is a $2\times 2$ matrix and given by 
\begin{equation*}
	V = \left( \begin{array}{cc}
	(2\gamma - \kappa_1 - \kappa_2)^2 ( \mathcal{L}_2^{-1} \varphi', \varphi' )_{L^2_{per}} & 0 \\
		0 & (\gamma - \kappa_2)^{-2} ( \mathcal{L}_1^{-1} \varphi,\varphi )_{L^2_{per}}
	\end{array}
	\right).
\end{equation*}
Using the results in Lemmas \ref{dif-norma} and \ref{LemmaL2-cn}, we obtain respectively that $(\mathcal{L}_1^{-1}\varphi,\varphi) = - \frac{d}{d\omega} \|\varphi\|_{L^2_{per}}^2 < 0$ and $(\mathcal{L}_2^{-1} \varphi',\varphi')_{L^2_{per}} > 0$, so that ${\rm n}(V) = 1$. On the other hand, since we have  ${\rm n}(\mathcal{L}) = 5$, we see that the difference ${\rm n}(\mathcal{L}) - {\rm n}(V) = 4$ is an even number. Therefore, we can not conclude that the multiple solution $\Phi = (\varphi,B\varphi,0,0)$ is spectrally stable or not.

\textbf{Case II:} $\gamma \in (\max\{\kappa_1,\kappa_2\},+\infty)$. In this specific case, it is possible to see that we can not decide the exact quantity of negative eigenvalues (see \eqref{tricotomy}) in order to apply the results in \cite{KapitulaKevrekidisSandstedeI} and \cite{KapitulaKevrekidisSandstedeII}. Therefore, we can not conclude a precise result of spectral stability for the multiple solution $\Phi = (\varphi,B\varphi,0,0)$.

\textbf{Case III:} $\gamma = 0$. Since ${\rm z}(\mathcal{L})=4$, we can proceed as in the third case of the last subsection to obtain
\begin{equation*}
	V = \left( \begin{array}{cccc}
		( \mathcal{L}_2^{-1} \varphi', \varphi' )_{L^2_{per}} & 0 & 0 & 0 \\
		0 & ( \mathcal{L}_2^{-1} \varphi',\varphi' )_{L^2_{per}} & 0 & 0 \\
		0 & 0 & ( \mathcal{L}_1^{-1} \varphi, \varphi )_{L^2_{per}} & 0 \\
		0 & 0 & 0 & ( \mathcal{L}_1^{-1} \varphi, \varphi )_{L^2_{per}}
	\end{array}
	\right).
\end{equation*}
Since by Lemmas \ref{dif-norma} and \ref{LemmaL2-cn} we have ${\rm n}(V) = 2$, the fact that ${\rm n}(\mathcal{L}) = 6$ gives us that the spectral stability result is also inconclusive.

\textbf{Case IV:} $\gamma = \kappa_1 = \kappa_2$. Since in this case ${\rm z}(\mathcal{L})=3$, we can proceed as in the fourth case of the last subsection with cnoidal profile instead of dnoidal profile to obtain $V$ as
\begin{equation*}
	V = \left( \begin{array}{ccc}
		( \mathcal{L}_2^{-1} \varphi',\varphi' )_{L^2_{per}} & 0 & 0 \\
		0 & (B^2 + 1)^2 ( \mathcal{L}_3^{-1} \varphi,\varphi )_{L^2_{per}} & 0 \\
		0 & 0 & (B^2+1)^2 ( \mathcal{L}_1^{-1} \varphi, \varphi )_{L^2_{per}}
	\end{array}
	\right).
\end{equation*}
Using Remark \ref{positivity} and Lemmas \ref{dif-norma} and \ref{LemmaL2-cn}, we have ${\rm n}(V) = 1$ and since ${\rm n}(\mathcal{L}) = 4$, we obtain that the difference ${\rm n}(\mathcal{L}) - {\rm n}(V)$ is an odd number. Consequently, the multiple solution $\Phi = (\varphi, B \varphi,0,0)$ is spectrally unstable.\\
\indent Summarizing the results above, we can conclude the statement of Theorem \ref{teo-2}.

\subsubsection{Spectral stability of cnoidal waves in the subspace of odd functions} Here we are going to answer some unclear points left behind concerning the spectral stability of periodic multiple solutions $\Phi$ with cnoidal profile. The arguments will be the same, but we need to pay attention with the spectral analysis established in the space $\mathbb{L}_{per,odd}^2$.

From now on, let us consider the periodic solution $\psi \in H^1_{per,odd}$ given by \eqref{cnoidal-2}. 

\textbf{Case I:} $\gamma \in (0,\min\{\kappa_1,\kappa_2\})$. Since in this case ${\rm z}(\mathcal{L}_{odd})=1$, we have that $V$ is given in a simple way as
\begin{equation}\label{V-2}
	V =  ( \mathcal{L}_1^{-1} \psi,\psi )_{L^2_{per}}.
\end{equation}

By Lemma \ref{dif-norma} we have ${\rm n}(V) = 1$ and since ${\rm n}(\mathcal{L}_{odd}) = 2$, we obtain that the difference ${\rm n}(\mathcal{L}_{odd}) - {\rm n}(V) = 1$ is an odd number. Therefore, we conclude that the multiple solution $\Psi = (\psi,B\psi,0,0)$ is spectrally unstable.

\textbf{Case II:} $\gamma \in (\max\{\kappa_1,\kappa_2\},+\infty)$. In this case, $V$ is an one-dimensional matrix given by same expression in \eqref{V-2}. Since we also have ${\rm n}(V) = 1$ and ${\rm n}(\mathcal{L}_{odd}) = 1$, we deduce that the multiple solution $\Psi = (\psi,B\psi,0,0)$  is spectrally stable.

\textbf{Case III:} $\gamma = 0$. We have ${\rm z}(\mathcal{L}_{odd})=2$ and the matrix $V$ is now given by
\begin{equation*}
	V = \left( \begin{array}{cc}
	 ( \mathcal{L}_1^{-1} \psi, \psi )_{L^2_{per}} & 0 \\
	 0 & ( \mathcal{L}_1^{-1} \psi, \psi )_{L^2_{per}}
	\end{array}
	\right).
\end{equation*}
By Lemma $\ref{dif-norma}$ and since $( \mathcal{L}_1^{-1} \psi, \psi )_{L^2_{per}}=- \frac{1}{2} \frac{d}{d\omega} \|\psi\|_{L^2_{per}}^2=-\frac{1}{2}\frac{d}{d\omega}||\varphi||_{L_{per}^2}^2 < 0$, we conclude ${\rm n}(V) = 2$. On the other hand, the fact ${\rm n}(\mathcal{L}_{odd}) = 2$ gives us that the difference ${\rm n}(\mathcal{L}_{odd}) -{\rm n}(V) =0$ and the multiple solution $\Psi = (\psi,B\psi,0,0)$  is then spectrally stable.

\textbf{Case IV:} $\gamma = \kappa_1 = \kappa_2$. Again, we have ${\rm z}(\mathcal{L}_{odd})=2$ and $V$ is given by
\begin{equation*}
	V = \left( \begin{array}{cc}
		 (B^2 + 1)^2 ( \mathcal{L}_3^{-1} \psi,\psi )_{L^2_{per}} & 0 \\
		 0 & (B^2+1)^2 ( \mathcal{L}_1^{-1} \psi, \psi )_{L^2_{per}}
	\end{array}
	\right).
\end{equation*}
By Remark \ref{positivity} and since $\mathcal{L}_3$ is positive, it follows that $( \mathcal{L}_3^{-1} \psi,\psi )_{L^2_{per}}>0$. On the other hand, by Lemma \ref{dif-norma} we obtain $( \mathcal{L}_1^{-1} \psi, \psi )_{L^2_{per}}=-\frac{1}{2}\frac{d}{d\omega}||\psi||_{L_{per}^2}^2<0$ and so, ${\rm n}(V) = 1$. Since ${\rm n}(\mathcal{L}_{odd}) = 1$, we deduce that the multiple solution $\Psi = (\psi,B\psi,0,0)$  is also spectrally stable.

Summarizing the above, we conclude the result in Theorem \ref{teo-odd}.

\begin{remark}\label{remstab1}
	
	As we have detailed in Remark $\ref{remstab}$, the abstract theories in \cite{grillakis1} and \cite{grillakis2} can be used to establish the orbital stability of periodic cnoidal waves in $\mathbb{H}_{per,odd}^1$ in the second case (Case II).
\end{remark}

\section{Spectral stability for the semi-trivial periodic solution with cnoidal profile}\label{section6}

An important aspect concerning the solution $(u,v)$ of the NLS system \eqref{NLS-system} is the existence of semi-trivial solutions
\begin{equation}\label{semitrivial-sol} 
	(u(x,t),v(x,t)) = (e^{i\omega t} \varphi(x), 0)
\end{equation}
of \eqref{NLS-system}. As we have already mentioned in the introduction, Hakkaev in \cite{Hakkaev} studied the spectral stability for the semi-trivial wave solution \eqref{semitrivial-sol} where $\varphi$ has a dnoidal profile. Our intention is to prove the spectral stability when $\varphi$ has a cnoidal profile.\\
\indent We follow as in Section 2. First, we substitute \eqref{semitrivial-sol} into \eqref{NLS-system} to obtain the ODE
\begin{equation}\label{ode-semitrivial} 
	-\varphi'' + \omega \varphi - \kappa_1 \varphi^3 = 0.
\end{equation}

\indent A similar result as determined in Theorem $\ref{cnoidalcurve}$ is now presented.

\begin{theorem}\label{prop-varphi-semitrivial}
	Let $L > 0$ be fixed. The equation \eqref{ode-semitrivial} has an $L$-periodic solution with cnoidal profile of the form 
	\begin{equation*}\label{phi-semitrivial}
		\varphi_\omega(x) = \frac{ \sqrt{2 \omega} k}{(2k^2 -1 )} \frac{1}{\sqrt{\kappa_1}} {\rm cn} \left( \frac{4 {\rm K}(k)}{L} x, k \right),
	\end{equation*}
where $\omega > 0$ is given by \eqref{omega-cnoidal} depends smoothly on $k \in \left(\frac{1}{\sqrt{2}},1\right)$ and $L>0$. In addition, the family
\begin{equation*}
	\omega \in (0,+\infty) \longmapsto \varphi = \varphi_\omega \in H^2_{per}([0,L])
\end{equation*}
of $L$-periodic solutions of \eqref{ode-semitrivial} depends smoothly on $\omega \in (0,+\infty)$.
\end{theorem}

The spectral problem to be studied in this case is $J \mathcal{L} u = \lambda u$, where $J$ is given by \eqref{J} and $\mathcal{L}$ is defined as 
\begin{equation}\label{L-semitrivial}
	\mathcal{L} = \left( \begin{array}{cccc}
		\mathcal{L}_1 & 0 & 0 & 0 \\
		0 & \mathcal{L}_3 & 0 & 0 \\
		0 & 0 & \mathcal{L}_2 & 0 \\
		0 & 0 & 0 & \mathcal{L}_4
	\end{array}
\right).
\end{equation}
Since $\mathcal{L}$ is a diagonal operator, we only need to analyze the spectral properties of the operators $\mathcal{L}_i$, $i=1,2,3,4$, where
\begin{align*}
\mathcal{L}_1 & = -\partial_x^2 + \omega - 3\kappa_1 \varphi^2, \\
\mathcal{L}_2 & = -\partial_x^2 + \omega - \kappa_1 \varphi^2, \\
\mathcal{L}_3 & = -\partial_x^2 + \omega - \gamma \varphi^2, \\
\mathcal{L}_4 & = -\partial_x^2 + \omega + \gamma \varphi^2.
\end{align*}

As determined in Subsection $\ref{cnoidalprof}$, we have that 
\begin{equation}\label{nL-semitrivial}
	{\rm n}(\mathcal{L}_1) = 2, \; \; {\rm n}(\mathcal{L}_2) = 1 \; \; \text{ and } \; \; {\rm n}(\mathcal{L}_4) = 0.
\end{equation}
In addition, 
\begin{equation}\label{kerL-semitrivial}
	{\rm Ker}(\mathcal{L}_1) = [\varphi'], \; \; {\rm Ker}(\mathcal{L}_2) = [\varphi] \; \; \text{ and } \; \; {\rm Ker}(\mathcal{L}_4) = \{0\}.
\end{equation}
Thus, the spectral analysis of the operator $\mathcal{L}$ changes according to the spectral analysis of the operator $\mathcal{L}_3$. Here, we consider the same cases for $\gamma$ as determined in \cite{Hakkaev}.

\textbf{Case I:} $\gamma \in (0,\kappa_1)$. Here, we have that
\begin{equation*}
	\mathcal{L}_1 < \mathcal{L}_2 < \mathcal{L}_3 < \mathcal{L}_4.
\end{equation*}
Using the comparison theorem, we can not obtain the exact values of ${\rm n}(\mathcal{L}_3)$ and ${\rm z}(\mathcal{L}_3)$ since we obtain three different different scenarios:
\begin{equation*}
	{\rm n}(\mathcal{L}_3) = 1 \text{ and } {\rm z}(\mathcal{L}_3) = 0, \; \; \; \;	{\rm n}(\mathcal{L}_3) = 0 \text{ and } {\rm z}(\mathcal{L}_3) = 1 \; \; \text{ or } \; \; 	{\rm n}(\mathcal{L}_3) = {\rm z}(\mathcal{L}_3) = 0.
\end{equation*}
Thus, the spectral analysis becomes inconclusive.

\textbf{Case II:} $\gamma = k_1$. In this case, it follows that $\mathcal{L}_2 = \mathcal{L}_3$ and we obtain
\begin{equation*}
	{\rm n}(\mathcal{L}) = 4 \text{ and } {\rm Ker}(\mathcal{L}) = \left[ (\varphi',0,0,0), (0,\varphi,0,0), (0,0,\varphi,0) \right].
\end{equation*}
To determine the spectral stability result, we need to obtain ${\rm n}(V)$. Since ${\rm z}(\mathcal{L})=3$, one has
\begin{equation*}
	V = \left( \begin{array}{ccc}
		( \mathcal{L}_2^{-1} \varphi',\varphi')_{L^2_{per}} & 0 & 0 \\
		0 & (\mathcal{L}_1^{-1} \varphi,\varphi)_{L^2_{per}} & 0 \\
		0 & 0 & (\mathcal{L}_4^{-1} \varphi, \varphi)_{L^2_{per}}
	\end{array}
\right).
\end{equation*} 
Since $\mathcal{L}_4$ is a positive operator, we obtain by Remark $\ref{positivity}$ that $(\mathcal{L}_4^{-1}\varphi,\varphi)_{L^2_{per}} > 0$. On the other hand, using Lemma \ref{LemmaL2-cn}, we also obtain $(\mathcal{L}_2^{-1}\varphi,\varphi)_{L^2_{per}} > 0$ and by Lemma \ref{dif-norma}, we deduce $(\mathcal{L}_{1}^{-1} \varphi,\varphi) < 0$. Gathering all informations, we conclude that ${\rm n}(V) = 1$ and the difference ${\rm n}(\mathcal{L}) - {\rm n}(V) = 4 - 1 = 3$ is an odd number. The periodic semi-trivial solution $\Phi=(\varphi,0,0,0)$ is then spectrally unstable concluding the desired result in Theorem \ref{teosemi}.

\textbf{Case III:} $\gamma \in (\kappa_1,3\kappa_1)$. In this case, we obtain the following inequality
\begin{equation*}
	\mathcal{L}_1 < \mathcal{L}_3 < \mathcal{L}_2 < \mathcal{L}_4.
\end{equation*}
From \eqref{nL-semitrivial}, \eqref{kerL-semitrivial} and using the comparison theorem, we have that ${\rm n}(\mathcal{L}_3) = 2$ and ${\rm Ker}(\mathcal{L}_3) = \{0\}$. Thus, 
\begin{equation*}
	{\rm n}(\mathcal{L}) = 5 \text{ and } {\rm Ker}(\mathcal{L}) = \left[ (\varphi',0,0,0), (0,0,\varphi,0) \right].
\end{equation*}
\indent To determine the spectral stability, we need to consider the matrix $V$ given by
\begin{equation*}
	V = \left( \begin{array}{cc}
		(\mathcal{L}_2^{-1} \varphi',\varphi')_{L^2_{per}} & 0 \\
		0 & (\mathcal{L}_1^{-1}\varphi,\varphi)_{L^2_{per}}
		\end{array}
	\right).
\end{equation*}
Applying again Lemmas \ref{dif-norma} and \ref{LemmaL2-cn}, we have that ${\rm n}(V) = 1$. Since the difference ${\rm n}(\mathcal{L}) - {\rm n}(V) = 4$ is even, we can not conclude the spectral stability.

\textbf{Case IV:} $\gamma = 3\kappa_1$. Here, we have ${\mathcal{L}_1} = \mathcal{L}_3$ and
\begin{equation*}
	{\rm n}(\mathcal{L}) = 5 \text{ and } {\rm Ker}(\mathcal{L}) = \left[ (\varphi', 0,0,0), (0,\varphi',0,0),(0,0,\varphi,0)\right].
\end{equation*}
Since ${\rm z}(\mathcal{L})=3$, the matrix $V$ becomes in this case
\begin{equation*}
	V = \left( \begin{array}{ccc}
		( \mathcal{L}_2^{-1} \varphi',\varphi')_{L^2_{per}} & 0 & 0 \\
		0 & (\mathcal{L}_1^{-1} \varphi,\varphi)_{L^2_{per}} & 0 \\
		0 & 0 & (\mathcal{L}_4^{-1} \varphi', \varphi')_{L^2_{per}}
	\end{array}
	\right).
\end{equation*}
Again, we can apply Remark \ref{positivity}, Lemma \ref{dif-norma} and Lemma \ref{LemmaL2-cn} to obtain ${\rm n}(V) = 1$. Since the difference ${\rm n}(\mathcal{L}) -n(V)=4$ is even, we can not decide about the spectral stability.

\subsection{Spectral stability of cnoidal waves in the subspace of odd functions} We consider the spectral stability for the semi-trivial wave solution $\Psi = (\psi,0,0,0)$ where $\psi$ in defined by 
$(\ref{cnoidal-odd})$. The reason for that is to fill the gaps left by our analysis performed in the case of the cnoidal profile $\Phi=(\varphi,0,0,0)$.

Since $\psi \in H^1_{per,odd}$, we can consider the linearized operator $\mathcal{L}$, $i=1,2,3,4$ restricted to the space $\mathbb{L}_{per,odd}^2$ as
\begin{equation*}
	\mathcal{L}_{odd} : \mathbb{H}^2_{per,odd} \subset \mathbb{L}^2_{per,odd} \rightarrow \mathbb{L}^2_{per,odd}
\end{equation*}
where $\mathcal{L}$ is given in \eqref{L-semitrivial} and also defined in $\mathbb{L}_{per,odd}^2$. Thus, as determined in Subsection \ref{subsection-odd}, we have that
\begin{equation}\label{nL-semitrivial-odd}
	{\rm n}(\mathcal{L}_{1,odd}) = 1, \; \; {\rm n}(\mathcal{L}_{2,odd}) = 0 \; \; \text{ and } \; \; {\rm n}(\mathcal{L}_{4,odd}) = 0.
\end{equation}
In addition, we get
\begin{equation}\label{kerL-semitrivial-odd}
{\rm Ker}(\mathcal{L}_{2,odd}) = [\psi] \; \; \text{ and } \; \; {\rm Ker}(\mathcal{L}_{1,odd}) = {\rm Ker}(\mathcal{L}_{4,odd}) = \{0\}.
\end{equation}

According to \eqref{nL-semitrivial-odd}, we need to observe that the number of eigenvalues of the operator $\mathcal{L}_{odd}$ is smaller when compared to the complete operator $\mathcal{L}$. This fact is useful to use again the comparison theorem. We shall describe better our intentions in the four cases ahead:\\
\indent \textbf{Case I:} $\gamma \in (0,\kappa_1)$. In this case, we have the inequality $\mathcal{L}_{1,odd} < \mathcal{L}_{2,odd} < \mathcal{L}_{3,odd} < \mathcal{L}_{4,odd}$. Using the comparison theorem and the informations in \eqref{nL-semitrivial-odd} and \eqref{kerL-semitrivial-odd}, we have that ${\rm n}(\mathcal{L}_{3,odd}) = 0$ and ${\rm Ker}(\mathcal{L}_{3,odd}) = \{0\}$. Thus, 
\begin{equation*}
	{\rm n}(\mathcal{L}_{odd}) = 1 \; \text{ and } \; {\rm Ker}(\mathcal{L}_{odd}) = \left[ (0,0,\psi,0) \right].
\end{equation*}
To evaluate the spectral stability, we need to see that ${\rm z}(\mathcal{L}_{odd})=1$, so that the matrix $V$ is given by
\begin{equation*}
	V = (\mathcal{L}_{1}^{-1} \, \psi,\psi)_{L^2_{per}}.
\end{equation*}
By Lemma \ref{dif-norma}, we have that ${\rm n}(V) = 1$ and the difference ${\rm n}(\mathcal{L}_{odd}) - {\rm n}(V)=0$. Thus, we conclude that the semi-trivial wave solution $\Psi = (\psi,0,0,0)$  is spectrally stable.\\
\indent \textbf{Case II:} $\gamma = \kappa_1$. Here, we have $\mathcal{L}_{2,odd} = \mathcal{L}_{3,odd}$. Using \eqref{nL-semitrivial-odd} and \eqref{kerL-semitrivial-odd}, we get
\begin{equation*}
	{\rm n}(\mathcal{L}_{odd}) = 1 \text{ and } {\rm Ker}(\mathcal{L}_{odd}) = \left[ (0,\psi,0,0), (0,0,\psi,0)  \right].
\end{equation*}
Since ${\rm z}(\mathcal{L}_{odd})=2$, the matrix $V$ is given by 
\begin{equation*}
	V = \left( \begin{array}{cc}
		(\mathcal{L}_{1}^{-1} \, \psi,\psi )_{L^2_{per}} & 0 \\
		0 & (\mathcal{L}_{4}^{-1} \, \psi,\psi)_{L^2_{per}}
	\end{array}
\right).
\end{equation*}
Since $\mathcal{L}_{4,odd}$ is positive, we obtain by Remark \ref{positivity} that $(\mathcal{L}_{4}^{-1} \, \psi,\psi)_{L^2_{per}}>0$. In addition, by Lemma \ref{dif-norma}, we see
$(\mathcal{L}_{1}^{-1} \, \psi,\psi )_{L^2_{per}}<0$, so that ${\rm n}(V) = 1$. Thus, the semi-trivial periodic wave $\Psi = (\psi,0,0,0)$  is spectrally stable.

\textbf{Case III:} $\gamma \in (\kappa_1, 3 \kappa_1)$. Here, we have the inequality $\mathcal{L}_{1,odd} < \mathcal{L}_{3,odd} < \mathcal{L}_{2,odd} < \mathcal{L}_{4,odd}$. Then, by the informations in \eqref{nL-semitrivial-odd}, \eqref{kerL-semitrivial-odd} and the comparison theorem, we obtain ${\rm n}(\mathcal{L}_{3,odd}) = 1$, ${\rm Ker}(\mathcal{L}_{3,odd}) = \{0\}$, ${\rm n}(\mathcal{L}_{odd}) = 2$ and ${\rm Ker}(\mathcal{L}_{odd}) = \left[ (0,0,\psi,0) \right]$. Again, since ${\rm z}(\mathcal{L}_{odd})=1$, the matrix $V$ is given by $V = (\mathcal{L}_{1}^{-1} \, \psi,\psi)_{L^2_{per}}$. By Lemma \ref{dif-norma}, we have ${\rm n}(V) = 1$, so that the semi-trivial solution $\Psi = (\psi,0,0,0)$ is spectrally unstable.

\textbf{Case IV:} $\gamma = 3 \kappa_1$. Here, we have $\mathcal{L}_{1,odd} = \mathcal{L}_{3,odd}$ and from \eqref{nL-semitrivial-odd} and \eqref{kerL-semitrivial-odd}, we have that ${\rm n}(\mathcal{L}_{odd}) = 2$ and ${\rm Ker}(\mathcal{L}_{odd}) = \left[ (0,0,\psi,0) \right].$ As we have determined in the last case, we also have ${\rm n}(V) = 1$ and the semi-trivial solution $\Psi = (\psi,0,0,0)$  is spectrally unstable.\\
\indent Summarizing the arguments above, we have proved Theorem $\ref{teooddsemi}$.

\begin{remark}\label{remstab2}
	
	As we have mentioned in Remarks $\ref{remstab}$ and $\ref{remstab1}$, the abstract theory in \cite{grillakis1} and \cite{grillakis2} can be used to establish the orbital stability of periodic cnoidal waves in $\mathbb{H}_{per,odd}^1$ in the first case (Case I). The orbital instability is deduced from the same work in the third and fourth cases (Cases III and IV).
\end{remark}

\section*{Acknowledgments}
F. Natali is partially supported by CNPq/Brazil (grant 303907/2021-5) and CAPES MathAmSud (grant 88881.520205/2020-01). G. E. Bittencourt Moraes is supported by Coordenação de Aperfeiçoamento de Pessoal de Nível Superior (CAPES)/Brazil - Finance code 001. A part of this work was developed when the first author was visiting the Departamento de Matemáticas of the Universidad Del Valle, Cali, Colombia. The first author would like to express his sincere thanks for their hospitality.

\end{document}